\documentclass[12pt]{amsart}
 \usepackage[]{amsmath, amsthm, amsfonts, amssymb}
\usepackage[mathscr]{eucal}
 \usepackage{hyperref}
 
 \newcommand {\C} {{\mathbb C}}
 \newcommand {\R} {{\mathbb R}}
 \newcommand {\Z} {{\mathbb Z}}
 \newcommand {\Q} {{\mathbb Q}}
  
 \newcommand {\bA}{{\mathbb A}}
 \newcommand {\PP} {{\mathbb P}}

 \newcommand {\FF} {{\mathbb F}}
  \newcommand {\F} {{\mathcal F}}

  \newcommand {\D} {{\mathcal D}}
   
  \newcommand {\h} {{\mathcal H}}
 \newcommand {\Ss} {{\mathcal S}}
  \newcommand {\Y} {{\mathcal Y}}
  
    \newcommand {\CC} {{\mathcal C}}
 \newcommand {\K} {{\mathcal K}}
  
  \newcommand {\M} {{\mathcal M}}
  
   \newcommand {\CO} {{\mathcal O}}

    \newcommand {\J} {{\mathcal J}}

   \newcommand {\del} {{\partial}}
 \newcommand {\X} {{\mathscr X}}

  \newcommand {\Div}{\text{\rm div}}

 \newcommand {\cl}{\text{\rm cl}}
 \newcommand {\CH}{\text{CH}} 
  \newcommand {\DIV}{\text{div}}

\newcommand{\alg}{\text{\rm alg}}

\newcommand{\Spec}{\text{\rm Spec}}
\newcommand{\I}{\text{\rm i}}

\newcommand{\ol}{\overline}
\newcommand{\ul}{\underline}

\newcommand{\Ext}{\text{\rm Ext}}
\newcommand{\MHS}{\text{\rm MHS}}
\newcommand{\Zar}{\text{\rm Zar}}
\newcommand{\pr}{\text{\rm pr}}
\newcommand{\bs}{{\backslash}}

\newcommand{\rat}{\text{\rm rat}}
\newcommand{\dgt}{\text{\rm dgt}}

 \newtheorem{thm}[subsection]{Theorem}
 
 \newtheorem{lemma}[subsection]{Lemma}
  
 \newtheorem{prop}[subsection]{Proposition}
  \newtheorem{conj}[subsection]{Conjecture}
 \newtheorem{defn}[subsection]{Definition}
 \newtheorem{rmk}[subsection]{Remark}
 \newtheorem{ex}[subsection]{Example}

\begin{document}

\title[ Height Pairing]{The business of Height Pairings}

\author{Souvik Goswami}

\address{632 Central Academic Building\\
University of Alberta\\
Edmonton, Alberta T6G 2G1, CANADA}

\email{souvik@ualberta.ca}

\author{James D. Lewis}

\address{632 Central Academic Building\\
University of Alberta\\
Edmonton, Alberta T6G 2G1, CANADA}

\email{lewisjd@ualberta.ca}

\thanks{Second author partially supported
by a grant from the Natural Sciences and Engineering Research Council 
of Canada}
\subjclass{14C25, 14C30, 14C35}
\keywords{Height pairing, 
Abel-Jacobi map, regulator, Deligne cohomology, Chow group}
\date{\today}

\renewcommand{\abstractname}{Abstract} 
\begin{abstract}  In algebraic geometry there is the notion
of a height pairing of algebraic cycles, which lies at the confluence
of arithmetic, Hodge theory and topology. After explaining
a motivating example situation, we introduce new directions in this subject.
\end{abstract}

\maketitle

\centerline{\em In celebration of Steven M. Zucker's 65th birthday.}

\centerline{\em A true pioneer in Hodge theory!}

 \tableofcontents{}

\section{Introduction}\label{S1}
From topology one has the notion of the local linking  number (or index)
of two curves in 3-space. Basically this determines locally how
many times a given curve winds around another (with orientation).
If one of curves bounds a membrane (we think of that membrane as a \underbar{pre}cycle, in the sense
that it's boundary is not zero), then the sum of these local links
can be interpreted as an intersection pairing.  {Paragraph 2.1. in \cite{Be3} comes
to mind about this.}
The height pairing of two
algebraic cycles is an algebraic generalization of this.  Here is an
example (see \cite{C-L}) of how we view a classical algebraic cycle as bounding
a precycle. Let $X$ be a projective algebraic manifold of dimension $d$
and $\{Z_{\alpha}\}$ a finite collection of irreducible subvarieties
of codimension $r-1$ in $X$. Let $f_{\alpha} \in \C(Z_{\alpha})^{\times}$,
and consider the precycle
\[
\xi_1' := \sum_{\alpha}(f_{\alpha},Z_{\alpha}).
\]
Put
\[
\xi_1 := \sum_{\alpha}\DIV_{Z_{\alpha}}(f_{\alpha}) \in z^r(X),
\]
where $z^r(X)$  are the cycles of codimension $r$ in $X$.
Note that by definition $\xi_1 \in z^r_{\rat}(X)$, the subgroup of cycles in $z^r(X)$ rationally equivalent
to zero.
\underbar{Alternate take}: Let $\square := \PP^1\bs \{1\}$. Then one can
interpret
\[
\xi_1' = \sum_{\alpha}\text{\rm graph}_{Z_{\alpha}\times\square}(f_{\alpha}) \in z^r(X\times\square),
\]
with
\[
\del(\xi_1') = \del_0(\xi'_1) - \del_{\infty}(\xi_1') = \xi_1.
\]
If $\xi_2\in z^{d-r+1}(X)$ is in general position with respect to $\xi_1$ (and $\xi_1'$), then
$|\xi_1|\cap |\xi_2|=\emptyset$; moreover
\[
\sum_{\alpha} \int_{Z_{\alpha}\cap \xi_2} \log |f_{\alpha}| \in \R,
\]
becomes the analog of the total linking index of $\xi_1$ and $\xi_2$. Now suppose
that $\xi_1 = 0$, i.e. $\del \xi_1'=0$. Then the real regulator   of the
``$K^{(r)}_1(X)$'' \underbar{cycle} $\xi_1'$, given by the formula,
\begin{equation}\label{Reg1}
R_{r,1}(\xi_1')\in H^{r-1,r-1}(X,\R) \simeq H^{d-r+1,d-r+1}(X,\R)^{\vee}, 
\end{equation}
\[
\omega\in H^{d-r+1,d-r+1}(X,\R) \mapsto \sum_{\alpha}\int_{Z_{\alpha}}\log|f_{\alpha}|\omega
\in \R,
\]
is well-defined (see \cite{Ja1}, or \cite{KLM} and the references cited there). If $\omega = [\xi_2]$ is algebraic, then
\[
R_{r,1}(\xi_1')(\omega) = \sum_{\alpha}\int_{Z_{\alpha}\cap \xi_2}\log|f_{\alpha}|.
\]
Finally, if $\omega=0$, e.g. $\omega = [\xi_2]$ where $\xi_2\in z^{d-r+1}_{\rat}(X)$,
then $R_{r,1}(\xi_1')(\omega)  = 0$. We deduce:

\begin{prop} We have a pairing
\[
\langle\ ,\  \rangle: z_{\rat}^r(X) \times
z_{\rat}^{d+1-r}(X)\to \R
\]
given by
\[
\langle \xi_1, \xi_2\rangle = R_{r,1}(\xi_1')(\xi_2) = \sum_{\alpha} \int_{Z_{\alpha}\cap \xi_2} \log |f_{\alpha}| \in \R,
\]
where $\xi_1\in z_{\rat}^r(X)$, $\xi_2\in
z_{\rat}^{d+1-r}(X)$ and $\xi_1'= \sum (f_{\alpha},Z_{\alpha})$ is a higher Chow precycle
whose divisor (boundary) is $\xi_1$.
It is easy to see that the pairing is well-defined,
i.e., it is independent of the exact choice of $\xi_1'$, since if
$\DIV(\xi_1' -\xi_1'') = 0$, then 
\[
R_{r,1}(\xi_1' -\xi_1'')(\xi_2) = 0
\]
as $\xi_2 \sim_{\rat} 0$. 
\end{prop}
{The projection formula holds trivially from the definition. That is, we
have}
\begin{prop}
Let $\pi: X\to Y$ be a flat surjective morphism between two smooth
projective varieties $X$ and $Y$. Then
$\langle \xi_1, \pi^* \xi_2\rangle = \langle \pi_* \xi_1,
\xi_2\rangle$ for all $\xi_1\in z_\text{rat}^r(X)$ and
$\xi_2\in z_\text{rat}^{d - r+1}(Y)$ with $|\pi_* \xi_1|\cap
|\xi_2| = \emptyset$.
\end{prop}
A little less obvious fact is that this pairing is symmetric. That is,
it has the following property which we will call the {\it reciprocity
property\/} of the pairing.

\begin{prop}\label{PKK}
For all $\xi_1\in z^r_{\rat}(X)$, $\xi_2\in
z_{\rat}^{d-r+1}(X)$ with $|\xi_1|\cap |\xi_2| = \emptyset$,
$\langle \xi_1, \xi_2\rangle = \langle \xi_2,\xi_1\rangle$.
\end{prop}

\begin{proof} 
Let $(f, D)$ and $(g, E)$ be the higher Chow precycles such that $\xi_1 =
\DIV(f)$ and $\xi_2 = \DIV(g)$.
We can assume, using some additional machinery \cite{Blo1}((Lemma 4.2), that
with regard to the pairs $(f,D)$, $(g,E)$, everything is
in ``general'' position. For notational simplicity, let
us assume that $D$ and $E$ are irreducible and meet properly
along an irreducible curve $C$. Let
\[
f_c := f\big|_C \in \C(C)^{\times}, \quad
g_c := g\big|_C\in \C(C)^{\times}.
\]
For every point $p\in C$, put
\begin{equation}\label{E251}
T_p\{f_c,g_c\} = 
(-1)^{\nu_p(f_c)\nu_p(g_c)}\biggl(\frac{f_c^{\nu_p(g_c)}}{g_c^{\nu_p(f_c)}}\biggr)_p.
\end{equation}
where $\nu_p(h)$ is the vanishing order of a function $h$
at $p$.
Since $|\xi_1| \cap |\xi_2| = \emptyset$, it follows that
\[
T_p\{f_c,g_c\} =
\begin{cases}
f_c^{\nu_p(g_c)}(p)&\text{if $\nu_p(g_c) \ne
  0$}\\
g_c^{-\nu_p(f_c)}(p)&\text{if $\nu_p(f_c) \ne 0$}\\
1&\text{otherwise.}
\end{cases}
\]
Then it is a consequence of Weil reciprocity:
\[
\prod_{p\in C}T_p\{f_c,g_c\} = 1,
\]
that
\[
\int_{D\cap \DIV(g)} \log|f| = \int_{E\cap \DIV(f)} \log|g|.
\]
Obviously, this is equivalent to $\langle \xi_1, \xi_2\rangle =
\langle \xi_2, \xi_1\rangle$. {The reader can consult chapters 2.2.2 and 2.2.3 of \cite{GS} for a generalization of this.}
\end{proof}

In addition, this pairing is also non-degenerate in the sense of detecting rational equivalence (see \cite{C-L} for details). {This pairing is 
a special case of  the complex Archimedean height pairing, well} presented in \cite{MS1}, and plays a role at ``infinity'' in \S\ref{S3}.
We will return to a generalization of this Archimedean height pairing in  \S\ref{S6}. 

\section{Notation and a breezy review of background material}\label{S2}

\noindent
$\bullet$ 
Unless otherwise specified, $X$ is a smooth projective variety of dimension $d$ defined
over a subfield $k\subseteq \C$, and $H^{\bullet}(X(\C))$ is singular cohomology, treating $X$
as a complex analytic space.

\noindent
$\bullet$ For a quasi-projective variety $W$ over a field ({{or more generally a noetherian and separated scheme $W$}}), $z^r(W)$ is the free abelian group
generated by subvarieties of codimension $r$ in $W$. {The Chow group of $W$ is defined as $\CH^r(W)=z^r(W)/z^r_{rat}(W)$, where $z^r_{rat}(W)$ is the subgroup of cycles rationally equivalent to zero. The rational Chow groups will be denoted by $\CH^r(W;\Q):=\CH^r(W)\otimes _{\Z}\Q$.}

\noindent
$\bullet$ 
Let $\bA\subseteq \R$ be a subring. The reader is assumed to have some familiarity with the
abelian category of $\bA$-MHS (mixed Hodge structures). Two excellent reference sources are \cite{B-Z} and \cite{Ja1}.
If $r\in \Z$, then the Tate twist $\bA(r) :=(2\pi\I)^r\bA$
is the (pure) Hodge structure with weight $-2r$ and Hodge type $(-r,-r)$. It is customary to make
the further assumption that $\bA\otimes_{\Z}\Q$ is a field, and we will assume this. The reasons
have to do with Deligne's observation (his $I^{p,q}$ decomposition theorem - a user-friendly 
explanation provided in \cite{St}) that the weight functor $W_{\bullet}$ is exact
(same for the Hodge filtration functor $F^{\bullet}$)\footnote{{Exactness is implied by strict compatibility 
which means that $h(F^rV_{1,\C}) = h(V_{1,\C})\cap F^rV_{2,\C}$
and $h(W_{\ell}V_{1,\bA\otimes\Q}) = h(V_{1,\bA\otimes\Q})\cap W_{\ell}V_{2,\bA\otimes\Q}$
for all $r$ and $\ell$. The idea is this. For any $\bA$-MHS $V$, $V_{\C}$
has a $\C$-splitting into a bigraded direct sum of  complex vector spaces $I^{p,q}
:= F^p\cap W_{p+q}\cap \big[\ol{F^q} \cap W_{p+q} + 
\sum_{i\geq 2}\ol{F^{q-i+1}}\cap W_{p+q-i}\big]$, where one shows that
$F^rV_{\C} = \oplus_{p\geq r}\oplus_qI^{p,q}$ and $W_{\ell}V_{\C} = \oplus_{p+q\leq\ell}I^{p,q}$.
Then by construction of $I^{p,q}$, one has $h(I^{p,q}(V_{1,\C}) \subseteq I^{p,q}(V_{2,\C})$. 
Hence $h$ preserves both the Hodge and complexified weight filtrations. Now use the fact that
$\bA\otimes\Q$ is a field to deduce that $h$ preserves the weight filtration over $\bA\otimes\Q$.}}. Let $V_1,\ V_2$ be $\bA$-MHS. Carlson \cite{Ca} was the
first to give an explicit description of $\Ext^1_{\bA-\MHS}(V_1,V_2)$ in terms of a ``torus'', with the consequence that
$\Ext^1_{\bA-\MHS}(V_1,-)$ is a right exact functor. If we assume for the moment that $\bA$-MHS has
enough injectives, then it is clear from formal homological algebra arguments that  $\Ext^n_{\bA-\MHS}(V_1,V_2) = 0$ form $n\geq 2$. In general,
one uses an Yoneda-Ext argument.  The vanishing of the higher $\Ext$'s was first proven by
Beilinson \cite{Be1}. 

\noindent
$\bullet$ 
Let's fix $\bA$ as per the above paragraph, and put, for $V$ a $\bA$-MHS,
$\Gamma(V) = \hom_{\bA-\MHS}(\bA(0),V)$, $J(V) = \Ext^1_{\bA-\MHS}(\bA(0),V)$.
For instance, if $\bA=\Q$, then the classical Hodge conjecture asserts that
$\Gamma\big(H^{2r}(X(\C),\Q(r))\big)$ is generated by the fundamental classes of 
cycles  $z^r(X;\Q) := z^r(X)\otimes \Q$.   The space $\Gamma\big(H^{2r}(X(\C),\Q(r))\big)$,
of dimension $M$ say, 
in untwisted form  is precisely $F^rH^{2r}(X(\C),\C) \cap
H^{2r}(X(\C),\Q)\simeq \bigoplus_{1}^M\Q(-r)$. In general, $F^rH^{i}(X(\C),\C) \cap
H^{i}(X(\C),\Q)$ need not be a Hodge structure, as first observed by Grothendieck (see \cite{Lew1}, \S 7)).
The (unique) largest Hodge structure in  $F^rH^{i}(X(\C),\C) \cap H^{i}(X(\C),\Q)$ is denoted
by $N_H^rH^i(X(\C),\Q)$. There is also a filtration by coniveau, denoted by 
\[
N^rH^i(X,\Q) \subseteq N_H^rH^i(X(\C),\Q).
\]
The (Grothendieck amended) general Hodge conjecture (GHC) asserts that
the aforementioned inclusion is an equality (the reader can again consult \cite{Lew1}(\S 7) for details).

\noindent
$\bullet$ If $V$ is a $\bA$-MHS, then by (\cite{Ca}, \cite{Ja2}),
\[
J(V) \simeq \frac{W_0V_{\C}}{F^0W_0V_{\C} + W_0V}.
\]
As an example for $X/\C$, $J\big(H^{2r-1}(X(\C),\Z(r))\big)$ denotes the $r$-th Griffiths jacobian, and
$J\big(H^{2r-2}(X(\C),\R(r))\big) \simeq H^{r-1,r-1}(X(\C),\R(r-1))$, the target space
(after incorporating twists{, viz., after multiplication by $(2\pi\I)^{r-1}$}) of $R_{r,1}$ in (\ref{Reg1}). Indeed, $J\big(H^{2r-2}(X(\C),\R(r-1))\big)$
is a version of real Deligne cohomology $H_{\D}^{2r-1}(X(\C),\R(r))$, where we consider $X/\C$  as a real variety via 
$X \to \Spec(\C)\to \Spec(\R)$ (see \cite{Ja1}).

\section{Intermezzo I}\label{IM1}

Steven Zucker's seminal work  \cite{Z}, the $L_2$-cohomology in the Poincar\'e metric associated to a
polarizable variation of Hodge structure over a base curve, turned out
to provide one instance of a $L_2$-cohomology interpretation  of a corresponding intersection cohomology, 
the coincidence in the general situation over an arbitrary base manifold $S$ with $\ol{S}$ K\"ahler, conjectured
by Deligne, and settled by the works of W. Schmid, A. Kaplan, and E. Cattani, following the development
of Schmid's $sl_2$-orbit theorem to several variables.\footnote{The reader is encouraged to consult \cite{C-K-S} for
more precise details concerning this discussion.}
In this part, we are interested in a lesser known result of Zucker's work, as it relates
to a global function field height pairing due to Beilinson \cite{Be3}, albeit in characteristic zero.
{We wish to make it clear that the construction here is simply an interpretation of section 1 in \cite{Be3},
from the point of view of the $L_2$-cohomology  in \cite{Z}.}
Start off with a diagram
\[
\begin{matrix}\label{IM1D1}
\X&\hookrightarrow&\ol{\X}\\
&\\
\rho\biggl\downarrow\ &&\quad\biggr\downarrow\ol{\rho}\\
&\\
C&{\buildrel j\over\hookrightarrow}&\ol{C}
\end{matrix}
\]
 where $\ol{C}$  is a smooth projective curve, $C$ affine,  $\ol{\rho}$ is proper, $\rho$  is smooth and proper, and all varieties
are smooth, defined over a field $k\subseteq \C$. Let $K = k(C) = k(\eta)$, $\eta\in C/k$ the generic point,  and set 
$X_K = \X_{\eta}$, the generic fiber. Note that
\[
\CH^r(\X_{\eta}) = \lim_{\buildrel \longrightarrow\over {U\subset C/k}}\CH^r(\rho^{-1}(U/k)),
\]
and that the cycle class map 
\[
\CH^r(\X_{\eta};\Q) \to H^{2r}(\X_{\eta}(\C),\Q(r)) :=  \lim_{\buildrel \longrightarrow\over {U\subset C/k}}H^{2r}(\rho^{-1}(U(\C))),\Q(r))
\]
is induced by
\[
\lim_{\buildrel \longrightarrow\over {U\subset C/k}}\big(\CH^r(\rho^{-1}(U/k);\Q)\to H^{2r}(\rho^{-1}(U(\C))),\Q(r))\big).
\]

\noindent{\underbar{Warning}}. The definition of $H^{2r}(X_K,\Q(r))$, which is commonly interpreted as $H^{2r}(X_K(\C),\Q(r))$, should
{\em not} be misconstrued as the same object as $H^{2r}(\X_{\eta}(\C),\Q(r))$, the latter defined by a limit process. 

\medskip
The affine Lefschetz theorem, 
the fact that $C$ is a curve, together with the (known degeneration of the) Leray spectral sequence (Deligne, but see
\cite{Z}(\S15) and the references cited there), tells us that the Leray filtration
\[
{H^{2r}(\X(\C),\Q(r)) = L_0\supset L_1\supset L_2\supset \{0\},}
\]
satisfies 
\[
L_0/L_1 = H^0(C,R^{2r}\rho_*\Q(r)),
\]
\[
 L_1/L_2 = H^1(C,R^{2r-1}\rho_*\Q(r)),
 \]
  \[
  L_2H^2(C,R^{2r-2}\rho_*\Q(r)) = 0,
 \]
 with same story for $H^{2r}(\X_{\eta}(\C),\Q(r))$, where we replace $C$ by $\eta$.
 It is {clear then that $\xi\in \CH^r_{\hom}(X_K;\Q) =  \CH^r_{\hom}(\X_{\eta};\Q) $ maps to zero in $H^0(\eta,R^{2r}\rho_*\Q(r))$
 by the Leray spectral sequence associated to $\rho$.}  Indeed,
 $\xi$ will have a spread cycle $\tilde{\xi}\in \CH^r(\ol{\X};\Q)$, with $\tilde{\xi}\big|_{\X} \mapsto 0\in H^0(C,R^{2r}\rho_*\Q(r))$.
 Thus  $\tilde{\xi}\big|_{\X}\in H^1(C,R^{2r-1}\rho_*\Q(r))$. Let $d = \dim X_K$, which is the relative dimension of
 the flat morphism $\ol{\rho}$. Observe that the product
 $
 H^1(C,R^{2r-1}\rho_*\Q(r)) \otimes H^1(C,R^{2d-2r+1}\rho_*\Q(d-r+1)) \xrightarrow{\cup} H^2(C,R^{2d}\rho_*\Q(d)) = 0,
 $
 is zero. {Indeed, to re-iterate, this is due to the affine Lefschetz theorem applied to a smooth affine $C$
 with cohomological degree $2 > 1 = \dim C$.} Notice however that $\ol{\X}$ is complete (and smooth){, and hence
 $H^{2r}(\ol{\X},\Q(r))$ is a pure Hodge structure of weight zero, viz., $H^{2r}(\ol{\X},\Q(r)) = W_0H^{2r}(\ol{\X},\Q(r))$. 
 Furthermore $[\xi]$ is the restriction of $[\tilde{\xi}]\in H^{2r}(\ol{\X},\Q(r))$.}
  Thus it is well known (rather implicit 
 after reading \cite{Z}(\S14) {and more to the point in  \cite{PM}}), that
for $\xi \in \CH^r_{\hom}(X_K;\Q)$, 
\[
[\xi]\in W_0H^1(\eta,R^{2r-1}\rho_*\Q(r)) = W_0H^1(C,R^{2r-1}\rho_*\Q(r))
\]
\[
= H^1(\ol{C},j_*R^{2r-1}\rho_*\Q(r)),
\]
the latter being the object of interest in \cite{Z}.
Note that the pairing (also being a pairing of intersection cohomologies)
\[
H^1(\ol{C},j_*R^{2r-1}\rho_*\Q(r))\otimes H^1(\ol{C},j_*R^{2d-2r+1}\rho_*\Q(d-r+1))
\]
\[
\xrightarrow{\cup} H^2(\ol{C},R^{2d}\rho_*\Q(d)) \simeq \Q,
\]
is non-degenerate \cite{Z}.
Now for $\xi_1\in \CH_{\hom}^r(X_K;\Q)$, $\xi_2\in \CH_{\hom}^{d-r+1}(X_K;\Q)$, we arrive at Beilinson's (global case) height pairing over 
the function field of a curve \cite{Be3}, viz.,
$
\langle\ ,\ \rangle : \CH_{\hom}^r(X_K;\Q)\otimes \CH_{\hom}^{d-r+1}(X_K;\Q) \to 
$
\[
H^1(\ol{C},j_*R^{2r-1}\rho_*\Q(r))\otimes H^1(\ol{C},j_*R^{2d-2r+1}\rho_*\Q(d-r+1)) \simeq \Q.
\]

\begin{rmk} For each closed point $v \in C$, Beilinson \cite{Be3} attaches a local linking number $\langle\ ,\ \rangle_v$, and shows that the global height pairing is the sum of local ones.
\end{rmk}

\section{The arithmetic scenario}\label{S3}

Now let $X$ be a smooth projective variety of dimension $d$, defined over a number field 
$k$ (i.e., $[k:\Q]<\infty $). Denote by $z^{\bullet }_{\hom}(X)$ the nullhomologous cycles. 
Under some assumptions, Beilinson in \cite{Be3} defined a height pairing
$$z^r_{\hom}(X;\Q)\times z^{d-r+1}_{\hom}(X;\mathbb{Q})\rightarrow \mathbb{R},$$
which factors through $z^{\bullet }_{\rat}(X;\mathbb{Q})${, viz.,
an induced  height pairing
\begin{equation}\label{BHT}
\langle\ ,\ \rangle _{\rm HT}:\CH^r_{\hom}(X;\mathbb{Q})\times \CH^{d-r+1}_{\hom}(X;\mathbb{Q})\rightarrow \mathbb{R}.
\end{equation}}
This pairing should have a number of conjectural properties, for example
{\begin{conj}\label{Beconj1}(Conjecture 5.4 (a) of \cite{Be3})
The height pairing is non-degenerate.
\end{conj}
\begin{conj}\label{Beconj2}(Hodge-index conjecture 5.5 of \cite{Be3})
Assume that a hard Lefschetz conjecture holds on null-homologous cycles (conjecture 5.3 of \cite{Be3}) and consider the primitive cycle decomposition. Let $L_X\in CH^1(X)$ be the class of a hyperplane section. Then the form $\langle \cdot , L^{d-2r+1}\cdot \rangle _{\rm HT}$ is definite of sign $(-1)^r$ on the primitive $r$-cycles for $r\leq \frac{d+1}{2}$. 
\end{conj}
These conjectures seem to mimick the nondegeneracy and the polarizing properties of the cohomology of $X$. For example, Conjecture \ref{Beconj1} is an analog of the non-degeneracy of
$$H^{2r-1}(X(\C),\mathbb{Q})\times H^{2d-2r+1}(X(\C),\mathbb{Q})\rightarrow \Q.$$}

It is instructive to explain the idea behind the pairing: {The ingredient comes from
\subsection{Arithmetic Chow groups}
References for this section are \cite{GS}, \cite{BGS}, and \cite{Ku}.
We will only provide a brief glimpse into this fascinating theory. Interested readers may consult the 
references cited above for details (especially \cite{BGS}). We begin with a motivating example (Chapter III of \cite{Neu}):
\begin{ex}\label{E3}
Consider a number field $k$ with the number ring $\mathcal{O}_k$. A prime $\wp $ of $k$ is a class of equivalent valuations of $k$. The non-Archimedean equivalence classes are called \textbf{finite} primes and accordingly the Archimedean ones are \textbf{infinite} primes. The infinite primes $\wp $ are obtained from the embeddings $\tau: k\hookrightarrow \C$. There are two sorts of these: the real primes, corresponding to the real embeddings, and the complex primes corresponding to the pairs of complex conjugate non-real embeddings. The finite primes will be denoted formally by $\wp $ and the infinite primes by $\wp _{\infty }$.\\

To each prime $\wp \in k$ (finite or infinite), we associate a canonical homomorphism
$$v_{\wp }: k^*\rightarrow \R$$
from the multiplicative group $k^*$ of $k$. If $\wp $ is a finite prime, then $v_{\wp }$ is a $\wp $-adic exponential valuation which is normalized by the condition $v_{\wp }(k^*)=\Z$. If $\wp $ is infinite, then $v_{\wp }$ is given by $v_{\wp }(a)=-\ln |\tau a|$, where $\tau $ is an embedding which defines $\wp $.\\

\textbf{Arakelov class group of $\mathcal{O}_k$:} The group $\widehat{Div}(\mathcal{O}_k)$- of (Arakelov) divisors is defined by elements of the form
$${D} :=\sum _{\wp }m_{\wp }\wp +\sum _{\wp _{\infty }}\lambda _{\infty }\wp _{\infty },$$
where $m_{\wp }\in \Z$ and $\lambda _{\infty }\in \R$, respectively. The principal divisors $\widehat{\mathcal{P}}(\mathcal{O}_k)$ are of the form
$$\sum _{\wp }v_{\wp }(\alpha )\wp +\sum _{\wp _{\infty }}(-\log|\alpha |_{\wp _{\infty }}),$$
where $|\alpha |_{\wp _{\infty }}=|\tau \alpha |$ if $\wp _{\infty }$ is real, and $|\alpha |_{\wp _{\infty }}=|\tau \alpha |^2$ if $\wp _{\infty }$ is complex. We define the \textit{Arakelov class group} of $\mathcal{O}_k$ as the quotient
$$\widehat{\rm Cl}(\mathcal{O}_k):= \widehat{Div}(\mathcal{O}_k)/\widehat{\mathcal{P}}(\mathcal{O}_k).$$
One can define a real number, called the \textit{degree} of a divisor ${D}$ as
$$\widehat{deg}({D}):=\sum _{\wp }m_{\wp }\log N_{\wp }+\sum _{\wp _{\infty }}\lambda _{\infty },$$
with $N_{\wp}=|\mathcal{O}_k/\wp |$. The degree of a principal divisor is zero by the product formula. Hence we get a well-defined (continuous) homomorphism
$$\widehat{deg}:\widehat{\rm Cl}(\mathcal {O}_k)\rightarrow \mathbb{R}.$$
\end{ex}
More generally, consider a regular, projective and flat scheme $\widetilde{X}\rightarrow S=\Spec(\mathcal{O}_k)$ of absolute dimension $d+1$. Such a scheme will be referred to as a regular arithmetic variety. Note that, from the definition $\widetilde{X}$ can be seen as a projective and flat scheme over $\Spec(\Z)$.\\
For a regular arithmetic variety $\widetilde{X}$, and any integer $r\geq 0$ we let $Z^r(\widetilde{X})$ be the free abelian group of cycles of codimension $p$ over $\widetilde{X}$. The set of complex points $\widetilde{X}(\C)$ of $\widetilde{X}$ can be identified with the disjoint union $\coprod _{\sigma : k\hookrightarrow \C}\widetilde{X}_{\sigma }(\C)$. Let $F_{\infty }: \widetilde{X}(\C)\rightarrow \widetilde{X}(\C)$ be the antiholomorphic involution coming from complex conjugation. We denote by $D^{r,r}(\widetilde{X}_{\R})$ the set of real currents in $D^{r,r}(\widetilde{X}(\C))$ ({with respect to a suitable action $F^*_{\infty }$ of $F_{\infty }$ on $D^{r,r}(\widetilde{X}(\C))$})\footnote{{A comment is in order here: In Arakelov setting, a current $\alpha \in D^{r,r}(\widetilde{X}(\C))$ is called real if $F^{*}_{\infty }(\alpha )= (-1)^r\alpha $ (see either section 3.2 of \cite{GS}) or 2.1 of \cite{BGS}).}}\\
Now, any cycle $Z\in {Z^r(\widetilde{X})}$ defines a current $\delta _Z\in D^{r,r}(\widetilde{X}_{\R})$ by integration on its set of complex points. A \textit{Green current} for $Z$ is any current $g\in D^{r-1,r-1}(\widetilde{X}_{\R})$ such that $dd^cg+\delta _Z$ is smooth (see \S 1 of \cite{BGS} for details and notations). Denote by $\widehat{Z}^r(\widetilde{X})$- the group of pairs $(Z, g_Z)$ where $Z\in Z^r(\widetilde{X})$ and $g_Z$ is a Green current for $Z$, with addition defined component-wise. Let $\widehat{R}^r(\widetilde{X})\subset \widehat{Z}^r(\widetilde{X})$ be the subgroup generated by pairs of the form $(0, \partial u+\overline{\partial }v)$, where $u$ and $v$ are currents of type $(r-2,r-1)$ and $(r-1,r-2)$ respectively ($\partial $ and $\overline{\partial }$ being suitable operations on the space of currents), and $(div(f), -\log|f|^2)$, where $f$ is a rational function on an integral subscheme $\widetilde{Y}\subset \widetilde{X}$ of codimension $r-1$, and $-\log|f|^2$ is the current on $\widetilde{X}(\C)$ obtained by restricting forms to the smooth part of $\widetilde{Y}(\C)$ and integrating against the $L^1$ function $-\log|f|^2$. Now define the \textit{arithmetic Chow group} of codimension $r$ as
{$$\widehat{\CH}^r(\widetilde{X}) = \widehat{Z}^r(\widetilde{X})/\widehat{R}^r(\widetilde{X}).$$}
\begin{rmk}
Arithmetic Chow groups can be defined for more general types of arithmetic varieties assuming {only} that the generic fibre is smooth (refer to \S 3.2 of \cite{GS} for details). {In case our arithmetic variety is $S=\Spec(\mathcal{O}_k)$ for the number ring $\mathcal{O}_k$ of $k$, the arithmetic Chow group $\widehat{\CH}^1(S){\cong }\widehat{\rm Cl}(\mathcal{O}_k)$.}\footnote{{One uses a special case of theorem 3.3.5, exact sequence (i) of \cite{GS}, setting $\widetilde{X}=S=\Spec(\mathcal{O}_k)$ (see section 3.4 of \cite{GS} for details).}}
\end{rmk}
We note down some crucial properties of arithmetic Chow groups:
\begin{itemize}
\item{ (Theorem 4.2.3 of \cite{GS}) There is a cup product of arithmetic Chow groups
$$\widehat{\CH}^r(\widetilde{X})\otimes \widehat{\CH}^s(\widetilde{X})\rightarrow \widehat{\CH}^{r+s}(\widetilde{X};\Q),$$
formally defined by the formula $[(Z_1,g_{Z_1})]\cdot [(Z_2,g_{Z_2})]=[(Z_1\cdot Z_2, g_{Z_1}\star g_{Z_2})]$, where $\star $ denotes the star product of Green currents (\S 1 of \cite{GS}).}
\vspace{0.4cm}
\item{(Theorem 3.6.1 and 4.2.3 of \cite{GS}) Let $f:\widetilde{X}\rightarrow \widetilde{Y}$ be a morphism of regular arithmetic varieties. Then there is a pull-back homomorphism $f^*: \widehat{\CH}^r(\widetilde{Y})\rightarrow \widehat{\CH}^r(\widetilde{X};\Q)$. It is multiplicative, i.e., given $\alpha \in \widehat{\CH}^r(\widetilde{Y})$ and $\beta \in \widehat{\CH}^s(\widetilde{Y})$, we have
$$f^*(\alpha \cdot \beta )=f^*(\alpha )\cdot f^*(\beta ).$$
Further if $f$ is proper, $f_k: \widetilde{X}_k\rightarrow \widetilde{Y}_k$ is smooth and $\widetilde{X}$, $\widetilde{Y}$ are equidimensional, then there is a push-forward homomorphism
$$f_*: \widehat{\CH}^r(X)\rightarrow \widehat{\CH}^{r-\delta }(Y),\hspace{0.1cm}(\delta :=\dim(X)-\dim(Y)),$$
satisfying the projection formula
$$f_*(f^*(\alpha )\cdot \beta )=\alpha \cdot f_*(\beta ).$$}
\end{itemize}}

\subsection{Arithmetic height pairing}
Let $\widehat {\CH}^*(\widetilde{X})$ be the arithmetic Chow theory as defined above. One can define an \textit{arithmetic degree} map as a push-forward
$$\widehat{deg}_{\widetilde{X}}: \widehat{\CH}^{d+1}(\widetilde{X})\rightarrow \widehat{\CH}^1(S){\cong }\widehat{\rm Cl}(\mathcal {O}_k)\rightarrow \mathbb{R},$$
{where $\widehat{\rm Cl}(\mathcal {O}_k)\rightarrow \mathbb{R}$ is the $\widehat{deg}$ map of Example \ref{E3}.\footnote{{We remark here that this $\widehat{deg}$ is really the composition of the push-forward morphism $\widehat{\CH}^1(S)\rightarrow \widehat{\CH}^1(\Spec(\Z))$ attached to the unique morphism $S\rightarrow \Spec(\Z)$, and the isomorphism $\widehat{\CH}^1(\Spec(\Z))\cong \R$ ( see 2.1.3 of \cite{BGS} for a detailed discussion of arithmetic degree maps).}}} Together with the arithmetic intersection, it defines a pairing
$$\widehat{\CH}^r(\widetilde{X};\mathbb{Q})\otimes \widehat{\CH}^{d-r+1}(\widetilde{X};\mathbb{Q})\rightarrow \widehat{\CH}^{d+1}(\widetilde{X};\mathbb{Q})\rightarrow \mathbb{R}.$$
For a smooth projective variety $X/k$, assume that it has a \textit{regular model} $\widetilde{X}$, i.e., a regular arithmetic variety which is projective and flat over $S:=\Spec(\mathcal{O}_k)$, together with an isomorphism $\widetilde{X}_k\cong X/k$. {Let's explore this a little bit further. We are considering a family $\widetilde{X}\rightarrow S$, where the generic fibre is isomorphic to $X/k$. This resembles the situation $\rho :\X\rightarrow C$ of \S\ref{IM1}. Now for each finite prime $\wp \in S$, we get a finite fibre $\widetilde {X}_{\wp }:=\widetilde {X}\times _S\Spec(\mathcal{O}_k/\wp \mathcal{O}_k)$, and for each embedding $\sigma :k\hookrightarrow \C$, we get a ``fibre at infinity'' $\widetilde{X}_{\sigma }:=\widetilde{X}\times _{\sigma }\C$. We think of the whole family (fibres over finite and infinite primes/embeddings) as a completion of $\widetilde{X}$ over $\overline{S}:=S\cup \{\sigma :k\hookrightarrow \C\}_{\sigma }$, 
resembling the situation $\ol{\rho }:\ol{\X}\rightarrow \ol{C}$ in \S\ref{IM1}. }
{\begin{rmk}\label{IVRm1}
The existence of a regular model for $X$ is a highly non-trivial problem. As a basic example, smooth projective curves have regular models (after possibly extending the base field). Apart from that, triple product of curves (Gross-Schoen) and abelian varieties (K\"unnemann) provides us with a large class of examples.
\end{rmk}}

With this set up and under a further assumption ((17) of \cite{Ku}), Beilinson's height pairing can be interpreted in light of arithmetic intersection
$$\widehat{\CH}^r(\widetilde{X};\mathbb{Q})\otimes \widehat{\CH}^{d-r+1}(\widetilde{X};\mathbb{Q})\rightarrow \widehat{\CH}^{d+1}(\widetilde{X};\mathbb{Q})\rightarrow \mathbb{R}.$$
This pairing {may a priori depend on the choice of $\widetilde{X}$. Since our primary aim is to detect non-trivial cycles, this choice is not a hinderance, once we have one.} To get a more earthly description, we stretch the analogy with Example \ref{E3} even further. Think of $\widetilde {X}$ as a projective scheme ${\rm Proj}\left(\mathcal{O}_k[X_0,\cdots ,X_N]/I\right)\rightarrow S$, for some homogeneous ideal $I$. The height pairing then is a sum of {non-Archimedean} and {Archimedean} parts
$$\langle \xi _1,\xi _2\rangle _{\rm HT}:=\sum _{\wp }\langle \xi _1, \xi _2\rangle _{\wp }+\sum _{\wp _{\infty }}\langle \xi _1,\xi _2\rangle _{\wp _{\infty }},$$
where (at least for finite primes $\wp $ of good reduction) $\langle \xi _1, \xi _2\rangle _{\wp }$ is given by
$$\log N_{\wp }\times \left(\#|(\xi _1\cap \xi _2)_{\wp }|, {\rm counting\hspace{0.1cm}multiplicities}\right),\
{\rm albeit\ heuristically!}$$

On the subgroup of cycles algebraically equivalent to zero, the height pairing is given by the N\'eron-Tate pairing
$$J^r_{\alg}(X)(\overline{k})_{\mathbb{Q}}\times J^{d-r+1}_{\alg}(X)(\overline{k})_{\mathbb{Q}}\rightarrow \mathbb{R},$$
where {$J^{r}_{\alg}(X)(\ol{k})_{\Q} :=\Phi _r(\CH^r_{\alg}(X/\ol{k};\Q))$}, and {$\Phi_r : \CH^r_{\hom}(X/k) \to J\big(H^{2r-1}(X(\C),\Z(r))\big)$} is the Griffiths
Abel-Jacobi map.
\begin{rmk}
Much like the notion of a height function, one can extend the height pairing for a smooth and projective $X$ defined over $\overline{\mathbb{Q}}$ (see 4.0.6 of \cite{Be3}).
\end{rmk}
{\begin{rmk}
The Archimedean part of the height pairing $\sum _{\wp _{\infty }}\langle \xi _1,\xi _2\rangle _{\wp _{\infty }}$ is given by the star product of green currents $g_{\xi _1}$ and $g_{\xi _2}$ (see \S 1 of \cite{MS1} for details). Restricting further to the subgroups $z^r_{rat}(X)$ and $z^{d-r+1}_{rat}(X)$, this Archimedean part of the height pairing is the one defined in \S \ref{S1} (up to factors). 
\end{rmk}}

\section{Bloch-Beilinson filtration}\label{S4}
It was first indicated in \cite{Blo2},  and later fortified by Beilinson, that for $X$ smooth and projective
over a field $k$, there should be a descending filtration
\[
F^0:=CH^r(X;\mathbb{Q})\supset F^1=CH^r_{\hom}(X;\mathbb{Q})\supset \cdots \supset F^r\supset \{0\},
\]
satisfying
\[
Gr_F^{\nu}\CH^r(X;\Q) := \frac{F^{\nu}\CH^r(X;\Q)}{F^{\nu+1}\CH^r(X;\Q})  \simeq \Ext^{\nu}_{\M\M}\big(\Spec(k),h^{2r-\nu}(X)(r))\big),
\]
where $\M\M$ is the conjectural category of mixed motives over $k$. A number of candidate
filtrations have been proposed (names will suffice) by Jannsen \cite{Ja3}, S. Saito \cite{SSa}, M. Saito/M. Asakura \cite{A}, Murre \cite{Mu}, Griffiths-Green
\cite{G-G}, Lewis \cite{Lew2},  Lewis/Kerr \cite{K-L},
Raskind \cite{Ra}, and so forth....  In  the case $k=\C$ we have seen that in the category of MHS,
$\Ext^{\bullet\geq  2}_{\MHS} = 0$, and yet $F^2\CH^r(X;\Q)$ need not be zero (Mumford \cite{M}, Bloch \cite{Blo2}, Lewis (op. cit. and \cite{Lew4}), Schoen
(see \cite{Ja2}), Roitman \cite{Ro}, Griffiths-Green \cite{G-G-P},...).
Even in the case where tr$\deg_{\Q}k = 1$, there are examples from some of the  references (op. cit.)
that $F^2\CH^r(X_k;\Q)\ne 0$. Indeed Beilinson and Bloch have independently conjectured the following:
\begin{conj}\label{BBC} {Let $X/\ol{\Q}$ be smooth and projective. Then
the Griffiths Abel-Jacobi map
$$\Phi _r: \CH^r_{\hom}(X/\ol{\Q};\mathbb{Q})\rightarrow  J(H^{2r-1}(X(\C),\mathbb{Q}(r))),$$
is injective.}
\end{conj}

\begin{rmk} \label{R1}{Assuming the classical Hodge conjecture, one can argue that $X$ in the conjecture can be replaced
by a smooth quasi-projective variety over $\ol{\Q}$. This follows from a weight filtered spectral sequence argument  \cite{K-L}(p. 371).}
\end{rmk}
For $k=\C$, the following theorem best summarizes one's expectations:
First consider fields $\ol{\Q} \subset K\subset \C$, where $K/\ol{\Q}$ is finitely generated. One
first constructs a filtration on $\CH^r(X_K;\Q)$. The ``lift'' from $K$ to $\C$ 
follows from:
\[
F^{\nu}\CH^r(X_{\C};\Q) = \lim_{\buildrel \longrightarrow
\over {K\subset \C}}F^{\nu}\CH^r(X_K;\Q)
\]

\begin{thm}[\cite{Lew2}]\label{LBB} Let $X/K$ be smooth projective 
of dimension $d$.
Then for all $r$, there is a filtration, 
\[
\CH^r(X_K;\Q) = F^0 \supset F^1\supset \cdots \supset F^{\nu}
\supset F^{\nu +1}\supset
\]
\[
 \cdots \supset F^r\supset F^{r+1}
= F^{r+2}=\cdots,
\]
which satisfies the following

\medskip
\noindent
{\rm (i)} $F^1 = \CH^r_{\hom}(X_K;\Q)$.

\medskip
\noindent
{\rm (ii)} $F^2 \subseteq\big( \ker \Phi_r  : \CH^r_{\hom}(X_K;\Q)
\rightarrow J\big(H^{2r-1}(X_K(\C),\Q(r))\big)\big)$.

\medskip
\noindent
{\rm (iii)}
$F^{\nu_1}\CH^{r_1}(X_K;\Q) \bullet F^{\nu_2}\CH^{r_2}(X_K;\Q)  \subset
F^{\nu_1 +\nu_2}\CH^{r_1+r_2}(X_K;\Q)$, where $\bullet$ is
the intersection product.

\medskip
\noindent
{\rm (iv)} $F^{\nu}$ is  preserved under the action of correspondences
between smooth projective varieties over $K$.

\medskip
\noindent
{\rm (v)} Assume that the K\"unneth components of the diagonal
class $[\Delta_X] = \oplus_{p+q = 2d}[\Delta_X(p,q)] \in
H^{2d}(X\times X,\Q(d)))$ are algebraic and defined over $K$.  Then
\[
\Delta_X(2d-2r+\ell+m,2r-\ell-m)_{\ast}\big\vert_{Gr_F^{\nu}\CH^r(X;\Q)}
= \delta_{\ell,\nu}\cdot \text{\rm Identity}.
\]
[If we assume the conjecture that homological and numerical equivalence coincide,
then (v) says that $Gr_F^{\nu}$ factors through the Grothendieck motive.]

\medskip
\noindent
{\rm (vi)} Let $D^r(X_K) := \bigcap_{\nu}F^{\nu}$.
 If Conjecture \ref{BBC} holds
for smooth quasi-projective varieties defined over $\overline{\Q}$ (vis-\`a-vis Remark \ref{R1}),
then $D^r(X_K) = 0$ (hence $D^r(X_{\C}) = 0$).
\end{thm}

It is instructive to briefly explain how this filtration comes about.
For $X/K$ smooth projective, one can find {a smooth quasi-projective} $\Ss/\ol{\Q}$ such that $\ol{\Q}(\Ss)$ is identified with $K$.
One can then spread out $X/K$ to a family $\rho : \X \to \Ss$, where
$\rho$ is a smooth and proper morphism of smooth quasi-projective
varieties over $\ol{\Q}$, and $X/K$ is the generic fiber. As a momentary digression, we
offer the reader an illuminating illustration of the notion of spreads:

\begin{ex}\label{EX000}{ Let
$$
Y/{\C} = {\Spec}\biggl\{\frac{{\C}[x,y]}{(\pi y^{2} + 
(\sqrt{\pi}+4)x^{3}+ \text{\rm e}x)}\biggr\}.
$$
$$
\Ss/{{\Q}} =  {\Spec}\biggl\{\frac{{\Q}[u,v,w]}{(u-v^{2})}\biggr\},
$$
Set:
$$
{\Y}_{\Ss} =   {\Spec}\biggl\{\frac{{\Q}[x,y,u,v,w]}{
\big(uy^{2} + (v+4)x^{3}+wx,u-v^{2})}\biggr\}
$$
The inclusion
\[
\frac{{\Q}[u,v,w]}{(u-v^{2})} \subset \frac{{\Q}[x,y,u,v,w]}{
\big(uy^{2} + (v+4)x^{3}+wx,u-v^{2})},
\]
defines a morphism ${\Y}_{\Ss}\to \Ss$,
as varieties over ${\Q}$. Let $\eta \in \Ss$, be the generic point. Then
$$
{\Q}(\eta) = \text{\rm Quot}\biggl(\frac{{\Q}[u,v,w]}{(u-v^{2})}\biggr).
$$
Note that the embedding
$$
{\Q}(\eta) \hookrightarrow {\C},\quad (u,v,w) \mapsto 
(\pi,\sqrt{\pi},\text{\rm e}),\
\Rightarrow
{\Y}_{\Ss,\eta}\times {\C} = Y/{\C}.
$$
We will have more to say about this in the next section.}
\end{ex}
 Now here is the key point.  Beilinson's absolute Hodge
cohomology  $H_{\h}$ \cite{Be1}, is a highly sophisticated cohomology theory with a number
of similar  properties to Deligne-Beilinson cohomology, with the advantage of incorporating
weights.  For our purposes here, we need the  short exact sequence { (p.2 of \cite{Be1}):}
\begin{equation}\label{ES}
J\big(H^{2r-1}(\X(\C);\Q(r))\big) \hookrightarrow H_{\h}^{2r}(\X(\C),\Q(r)) \twoheadrightarrow \Gamma\big(H^{2r}(\X(\C),\Q(r))\big).
\end{equation}
There is a cycle class map $\cl_r : \CH^r(\X/\ol{\Q};\Q) \to H_{\h}^{2r}(\X(\C),\Q(r))$, and according to 
Conjecture \ref{BBC} and Remark \ref{R1}, 
one anticipates that $\cl_r$ is injective. The lowest weight part, $\ul{H}^{2r}_{\mathcal H}(\X(\C),\Q(r))
\subset {H}^{2r}_{\mathcal H}(\X(\C),\Q(r))$ is given by the image
$H^{2r}_{\mathcal H}(\ol{\X}(\C),\Q(r)) \to H^{2r}_{\mathcal H}(\X(\C),\Q(r))$,
where $\ol{\X}/\ol{\Q}$ is a smooth compactification of $\X/\ol{\Q}$. 
Note that {$\CH^r(\ol{\X}/\ol{\Q}) \to \CH^r(\X/\ol{\Q})$}, is surjective{;
likewise there is a cycle class map $\CH^r(\ol{\X};\Q) \to H^{2r}_{\h}(\ol{\X},\Q(r))$.}
Thus we conjecturally have an injection
\[
\cl_r : \CH^r(\X/\ol{\Q};\Q) \to \ul{H}_{\h}^{2r}(\X(\C),\Q(r)).
\]
The  filtration $\F^\nu\CH^r(\X/\ol{\Q};\Q)$ is given by the
pullback of the $\nu$-th Leray filtration of $\rho$ on $\ul{H}^{2r}_{\h}(\X(\C),\Q(r))$, to
$\CH^{r}(\X/\ol{\Q};\Q)$.  (For an excellent motivic description of the Leray filtration, the
reader should consult \cite{Ar}.) Let $\eta_{\Ss}$
be the generic point of $\Ss$, and put $K := \ol{\Q}(\eta_{\Ss})$, and note
that the sequence in (\ref{ES}) remains exact at the generic point, by properties of direct limits. Write
$X_K := \X_{\eta_{\Ss}}$. The injectivity of $\cl_r$ {passes to the
generic point of $\Ss$, viz., $\cl_r : \CH^r(\X_{\eta_{\Ss}};\Q) \hookrightarrow \ul{H}_{\h}^{2r}(\X_{\eta_{\Ss}};(\C),\Q(r))$,
 leading to a filtration $\{F^{\nu}\CH^r(X_K;\Q)\}_{\nu\geq 0}$.}
Thus following  \cite{Lew2} we introduced a 
decreasing filtration $F^\nu\CH^r(X_K;\Q)$, with the
property that   $Gr_F^\nu\CH^r(X_K;\Q) \hookrightarrow 
E_{\infty}^{\nu,2r-\nu}(\eta_{\Ss})$, where $E_{\infty}^{\nu,2r-\nu}(\eta_{\Ss})$
is the $\nu$-th graded piece of the Leray filtration associated to $\rho$ on
$\ul{H}^{2r}_{\mathcal H}(\X_{\eta_{\Ss}},\Q(r))$ 
It is proven in \cite{Lew2} that the term $E_{\infty}^{\nu,2r-\nu}(\eta_{\Ss})$
fits in a short exact sequence:
$$
0\to \ul{E}_{\infty}^{\nu,2r-\nu}(\eta_{\Ss}) \to E_{\infty}^{\nu,2r-\nu}(\eta_{\Ss})
\to \ul{\ul{E}}_{\infty}^{\nu,2r-\nu}(\eta_{\Ss}) \to 0,
$$
where
\[
\ul{\ul{E}}_{\infty}^{\nu,2r-\nu}(\eta_{\Ss}) =
\Gamma(H^\nu(\eta_{\Ss},R^{2r-\nu}\rho_{\ast}\Q(r))),\quad \ul{E}_{\infty}^{\nu,2r-\nu}(\eta_{\Ss}) =
\]
\[
\frac{J(
W_{-1}H^{\nu-1}(\eta_{\Ss,}R^{2r-\nu}\rho_{\ast}\Q(r)))}{\Gamma(Gr_W^0H^{\nu-1}
(\eta_{\Ss},R^{2r-\nu}\rho_{\ast}\Q(r)))}
\subset J(H^{\nu-1}(\eta_{\Ss},R^{2r-\nu}\rho_{\ast}\Q(r))).
\]
Here the latter inclusion is a result of the short exact sequence:
\[
0\to W_{-1}H^{\nu-1}(\eta_{\Ss},R^{2r-\nu}\rho_{\ast}\Q(r)) \to
W_0H^{\nu-1}(\eta_{\Ss},R^{2r-\nu}\rho_{\ast}\Q(r))
\]
\[
\to Gr_W^0H^{\nu-1}(\eta_{\Ss},R^{2r-\nu}\rho_{\ast}\Q(r))\to 0.
\]
 {We attend to (vi). The idea comes from \cite{Ra}; however, as noted in \cite{Ra}, it goes
 back to  a hard Lefschetz argument due to Beauville (see \cite{Lew2}). This will imply that
$D^r(X) := D^{r}(X_K) = 0$ under Conjecture \ref{BBC} for smooth quasi-projective varieties. 
It is instructive to the reader to explain this argument. 
It suffices to show that 
\begin{equation}\label{E000}
\lim_{{\buildrel\longrightarrow\over{U\subset \Ss/\ol{\Q}}}}F^{r+1}\CH^r(\rho^{-1}(U);\Q)) =:  F^{r+1}\CH^r(X_K;\Q))
\end{equation}
\[
 =
\lim_{{\buildrel\longrightarrow\over{U\subset \Ss/\ol{\Q}}}}F^{r+j}\CH^r(\rho^{-1}(U);\Q)) =: F^{r+j}\CH^r(X_K;\Q),\ {\rm for \ all}\ j\geq 1.
\]
Let $L_X$ be the operation of cupping with the hyperplane class of the fibers of $\rho$, and 
\[
\psi_{r+j}: F^{r+j}\CH^r(\rho^{-1}(U);\Q))  \to E_{\infty}^{r+j,r-j},
\]
 the natural map. There is a commutative diagram
 \[
 \begin{matrix}
F^{r+j}\CH^r(\rho^{-1}(U);\Q))&\xrightarrow{\psi_{r+j}}&E_{\infty}^{r+j,r-j}\\
\quad\quad\biggl\downarrow L_X^{d -r+j}&&\quad\wr \biggr\downarrow L_X^{d -r+j}\\
F^{r+j}\CH^{d+j}(\rho^{-1}(U;\Q))&\to&E_{\infty}^{r+j,2d-r+j}
\end{matrix}
\]
Since $F^{r+j+1}\CH^{r+j}(\rho^{-1}(U);\Q)) = \ker \psi_{r+j}$, and that 
\[
\lim_{{\buildrel\longrightarrow\over{U\subset \Ss/\ol{\Q}}}}CH^{d+j}(\rho^{-1}(U);\Q)) = \CH^{d+j}(X_K;\Q) = 0,
\]
 for $j\geq 1$, as $\dim X_K = d$,
it follows that (\ref{E000}) holds, and we're done.}

\section{{The business of spreads}}

{Consider the smooth elliptic curve 
\[
E_{\C} :=  {\rm Proj}\biggl(\frac{\C[z_0,z_1,z_2]}{(2{\rm e}z_0z_2^2 -\pi z_1^3 +\sqrt{\pi}z_1z_0^2 +\sqrt{3}\I z_0^3)} \biggr)\subset \PP^2(\C)
\]
An analytic geometer may view this as a compact Riemann surface endowed with the analytic
topology. If for the moment we view $E$ as a prototypical projective algebraic manifold, then one 
key distinguishing feature that $E$ (or for that matter any complex algebraic variety) 
has over general complex manifolds, is that it can be arrived at via base extension
from a smaller subfield $K\subset \C$. There are lots of choices for $K$, but it is customary to think of it as 
finitely generated over $\Q$. To muddy the water a bit, let's consider $\xi = \{\sqrt{5}z_0^4z_1 + \I z_2^5 = 0\} \cap E_{\C} \in \CH^1(E_{\C})$.
 In this case, let's choose $K = \Q(\sqrt{\pi},{\rm e},\I,\sqrt{3},\sqrt{5})$, and define
\[
E_K :=  {\rm Proj}\biggl(\frac{K[z_0,z_1,z_2]}{(2^{-1}{\rm e}z_0z_2^2 -\pi z_1^3 +\sqrt{\pi}z_1z_0^2 +\sqrt{3}\I z_0^3)} \biggr)\subset \PP^2_K.
\]
By base change, we have $E_{\C} = E_K\times_K\C := E_K\times_{\Spec(K)}\Spec(\C)$. Now $K$ itself can be representative of the
process of evaluation of a general point over $\Q$. Let 
\[
\Ss = \Spec\biggl(\frac{\Q[u,v,w,t,s]}{(w^2+1,t^2-3,s^2-5)}\biggr).
\]
Let $\eta_{\Ss}\in \Ss/\Q$ be the generic point. Note  that by definition $\Q(\eta_{\Ss}) = \Q(\Ss)$ and that the evaluation map
\begin{equation}\label{EQ000}
\Q(\Ss) \hookrightarrow \C,\ (u,v,w,t,s) \mapsto  p := ({\rm e},\sqrt{\pi}, \I,\sqrt{3},\sqrt{5}) \in \Ss(\C),
\end{equation}
identifies $\Q(\Ss)$ with $K$. Any other point $q\in \Ss(\C)$ for which evaluation defines an embedding
$K\hookrightarrow \C$ is called a general point of $\Ss$. Now consider the quasi-projective variety $\mathcal{E}_{\Q}$
defined by
\[
 \left\{ \begin{matrix} 2^{-1}uz_0z_2^2 -v^2 z_1^3 
+v z_1z_0^2+ wtz_0^3 = 0\\
w^2+1= t^2-3= s^2-5 = 0\end{matrix}\right\} \subset \PP_{\Q}^2\times_{\Q}\Spec(\Q[ u,v,w,t,s]).
\]
Likewise, $\xi$ has an obvious spread $\tilde{\xi}$ given by
\[
 \left\{ \begin{matrix} 2^{-1}uz_0z_2^2 -v^2 z_1^3 
+v z_1z_0^2+ wtz_0^3 = 0\\
sz_0^4z_1 +wz_2^5 = 0\\
w^2+1= t^2-3= s^2-5 = 0\end{matrix}\right\}.
\]
As in Example \ref{EX000}, there is a morphism
 $\rho : \mathcal{E}_{\Q} \to \Ss_{\Q}$.
Then  $\tilde{\xi}\in \CH^1(\mathcal{E}_{\Q})$.
Indeed $\tilde{\xi}_{\eta_{\Ss}} = \xi\in \CH^1(E_K;\Q)$, where $\mathcal{E}_{\eta_{\Ss}}$ is identified with $E_K$,
 under the embedding given in (\ref{EQ000}).  Let us also view $\rho : \mathcal{E}(\C) \to \Ss(\C)$ as the induced map of
 complex spaces.
 The datum associated to the Leray sheaf $R^{\bullet}\rho_*\Q$ amounts to an arithmetic variation of  Hodge structure, and these ideas
 have played  a big role in constructing algebraic invariants associated to Chow groups of algebraic cycles, as for example  seen in the
 previous section. The  reader should also consult \cite{A}, \cite{G-G} and \cite{Lew4} as further exploitation of these ideas. A different
 line of enquiry involving spreads can be found in \cite{V}.
 Finally one can also spread $\mathcal{E}$ over $\Z$, by including the equation $2x-1=0$. This leads to an arithmetic scheme over ${\Z}$
 where the business of height pairings can be addressed.}

 \section{Intermezzo II}
 
 At this point, it should be reasonably clear to the reader that the notion of
 a height pairing  of the form 
\begin{equation}
F^{\nu}\CH^r(X_K;\Q) \times F^{\nu}\CH^{d-r+\nu}(X_K;\Q) \to \R,
\end{equation}
generalizing (\ref{BHT}), and providing a ``polarization'' on ``primitive'' pieces of $Gr_F^{\nu}\CH^r(X;\Q)$,
much the same way as with the Hodge-Riemann bilinear relations on the primitive cohomology of
a projective algebraic manifold, should exist.  Here $K$ is finitely generated over $\ol{\Q}$. As we will see below,
there is the technical requirement that $K$ have transcendence degree $\nu-1$
over $\Q$, $\nu\geq 1$. Unfortunately, a proof of such a pairing seems elusive at this given time,
and so we were forced to make further concessions (\S\ref{S5}).

\medskip
The relevance of these ideas  should be clear. The idea of attaching a conjecturally non-degenerate pairing
on graded pieces of the Bloch-Beilinson filtration is a unique new idea that
is at the cross roads of arithmetic, Arakelov geometry and Hodge theory. At the heart of the notion
of a height pairing of two cycles, is the idea of ``spreading'' a cycle out so as to form an intersection
pairing, very similar to the aforementioned idea of defining a linking number of two disjoint curves in 3-space, where one
curve bounds a membrane, thus creating an intersection number with the other curve. In (co-)homology theory,
it is often the case that to determine whether a (co-)cycle is non-zero, is via an intersection/cup product with
a complementary dimensional cycle. The ``definite'' properties of the N\'eron-Tate pairing (and
conjecturally that of the Beilinson pairing) should convince one that this technology may lead to similar
role in detecting the non-triviality of a \underbar{specific}  ``interesting'' algebraic cycle.

\section{A new pairing}\label{S5}
{\em Throughout this section, we will assume Conjecture \ref{BBC} and the GHC.}

 \medskip
If $\ol{\Q}$ of the previous section is replaced by $K$, a field of finite transcendence degree over 
$\overline{\mathbb{Q}}$, then as alluded to earlier, Conjecture \ref{BBC} is false. However
as indicated in  \S\ref{S4},  the notion of a conjectural Bloch-Beilinson filtration
involves spreads, which is key to a generalized pairing. We will continue with the notation
of \S\ref{S4}, with $K = \ol{\Q}(\Ss)$ finitely generated over $\ol{\Q}$.
Let $Gr^{\nu }_FCH^r(X_K;\mathbb{Q})$ denote the graded pieces of the filtration, we 
have a non-canonical motivic decomposition (albeit $Gr^{\nu }_F\CH^r(X_K;\mathbb{Q})$ is unique)
\[
\CH^r(X_K;\mathbb{Q}) = \bigoplus_{\nu\geq 0}\Delta_{X_K}(2d-2r+\nu,2r-\nu)_*\CH^r(X_K;\Q)
\]
\[
\simeq \bigoplus _{\nu \geq 0}Gr^{\nu }_F\CH^r(X_K;\mathbb{Q}),
\]
much like the Hodge decomposition of the {de Rham} cohomology of $X$.\\
Now if $X/\overline{\mathbb{Q}}$ is smooth projective, in light of Conjecture \ref{BBC}, Beilinson's 
height pairing could be interpreted as a pairing on $Gr^1_F$. In \cite{S-G}, we obtained the following 
extension of Beilinson's pairing for higher graded pieces:
\begin{thm}\label{IIITh1}
Let $X/\overline {\mathbb {Q}}$ be a smooth projective variety of dimension $d$ and 
let $K/\overline {\mathbb {Q}}$ be a finitely generated overfield of transcendence degree $\nu -1$, 
where $\nu \geq 1$ is an integer. Then there exists a pairing
$$\langle\hspace{0.1cm},\hspace{0.1cm}\rangle_{\rm HT}^{{\nu}} : Gr^{\nu }_{F}
\CH^r(X_{K};\mathbb {Q})\times Gr^{\nu }_{F}\CH^{d-r+\nu }(X_{K};\mathbb {Q})\rightarrow \mathbb{R},$$
extending Beilinson's height pairing.
\end{thm}

\begin{proof} (Sketch only.)
First note that $K\cong \overline {\mathbb {Q}}(S)$ where $S/\overline {\mathbb {Q}}$ is a smooth projective variety of dimension $\nu -1$ and let $\eta _{S}$ be the generic point of $S$. In this case $\rho : \X \to \Ss$ is given by $Pr_S:S\times X\rightarrow S$. We have the short exact sequence at the generic point
$$0\rightarrow \underline {E}^{\nu , 2r-\nu }_{\infty }(\eta _S)\rightarrow E^{\nu ,2r-\nu }_{\infty }(\eta _S)\rightarrow \underline{\underline{E}}^{\nu , 2r-\nu }_{\infty }(\eta _S)\rightarrow 0,$$
where
$$\underline{\underline{E}}^{\nu ,2r-\nu }_{\infty }(\eta _S)=\Gamma \left(H^{\nu }(\eta _S,R^{2r-\nu }\rho _*\mathbb {Q}(r))\right)=0\hspace{0.1cm}$$
by the affine Lefschetz theorem and
$$\underline {E}^{\nu ,2r-\nu }_{\infty }(\eta _S)= \frac{J\left(W_{-1}\left(H^{\nu -1}(\eta _S,\mathbb {Q})\otimes H^{2r-\nu }(X,\mathbb {Q})\right)(r)\right)}{\Gamma \left(Gr^0_W\left(H^{\nu -1}(\eta _S,\mathbb {Q})\otimes H^{2r-\nu }(X,\mathbb {Q})\right)(r)\right)}\hspace{0.1cm}.$$
We also have $Gr^{\nu }_FCH^r(X_K;\mathbb{Q})\hookrightarrow \underline {E}^{\nu ,2r-\nu }_{\infty }(\eta _S)$.\\

The following two propositions are key to the proof.
\begin{prop}[ \cite{Lew3}]\label{LLL}
There is an injective map
$$Gr^\nu _{F}\CH^r(X_{K};\mathbb {Q})\hookrightarrow J(H_{0}).$$
Here $J(H_{0})$ denotes the jacobian of the pure Hodge structure $H_{0}$ defined by
$$H_{0} :=\left(\frac{H^{\nu -1}(S,\mathbb {Q})}{N^1_{\overline {\mathbb {Q}}}H^{\nu -1}(S,\mathbb {Q})} \otimes\frac {H^{2r-\nu }(X,\mathbb {Q})}{N^{r-\nu +1}_{H}H^{2r-\nu }(X,\mathbb {Q})}\right)(r)\hspace {0.1cm}.$$
\end{prop}

Rather than explain the details of the proof of Proposition \ref{LLL}, the main
philosophical point is the expectation\footnote{This is also apparent in  the work of Shuji Saito \cite{SSa}.} that
\[
\Ext^{\nu}_{\M\M}\big(\Spec(K),h^{2r-\nu}(X_K)(r)\big) \simeq
 \Ext^{\nu}_{\M\M}\biggl(\Spec(K),\frac{h^{2r-\nu}(X_K)(r)}{N_K^{r-\nu+1}}\biggr).
 \]
Unfortunately, any attempt to extend Proposition \ref{LLL} beyond $X = X_{\ol{\Q}}$, viz., to $X_K$, involving a twisted
spread $\X \to \Ss$, $\ol{\Q}(\Ss)=K$,  seems highly non-trivial. 
Next,
\begin{prop}[Lewis]
There is a surjective map
$$\CH^r_{\hom}((S\times X)_{\overline{\mathbb{Q}}};\mathbb{Q})\twoheadrightarrow Gr^{\nu }_F\CH^r(X_K;\mathbb{Q})$$
given by the projector $\Delta _S\otimes \Delta _X(2d-2r+\nu ,2r-\nu )$.
\end{prop}

\begin{proof}
{First of all we observe that 
$\CH^r(S \times_{\overline\Q} X) \twoheadrightarrow  \CH^r(X_K)$ is surjective.
Therefore by Theorem \ref{LBB}(v), the composite involving the full Chow group:
\[
\CH^r(S \times_{\overline\Q} X;\Q)  \xrightarrow{({\rm Id}\otimes \Delta_{X/\ol{\Q}}
(2d-2r+\nu,2r-\nu))_{\ast}}
\]
\[ \Delta_{X_K}(2d-2r+\nu,2r-\nu)_{\ast}\CH^r(X_K;\Q) \simeq
Gr_F^{\nu}\CH^r(X_K;\Q),
\]
is surjective. 
Now  for any smooth affine subvariety $U \subset S/\ol{\Q}$ of dimension $< \nu$,
the affine Lefschetz theorem implies that 
$H^{\nu}(U,\Q)=0$. Applying the K\"unneth formula to $H^{2r}(U\times X,\Q(r))$,
it follows that 
\[
({\rm Id}\otimes \Delta_{X/\ol{\Q}}(2d-2r+\nu,2r-\nu))_{\ast}
\CH^r(U \times_{\overline\Q} X;\Q)  \mapsto 0 \in H^{2r}(U\times X,\Q(r)),
\]
and hence accordingly, 
\[
\CH_{\hom}^r(U \times_{\overline\Q} X;\Q)  \xrightarrow{({\rm Id}\otimes \Delta_{X/\ol{\Q}}
(2d-2r+\nu,2r-\nu))_{\ast}}
\]
\[ \Delta_{X_K}(2d-2r+\nu,2r-\nu)_{\ast}\CH^r(X_K;\Q) \simeq
Gr_F^{\nu}\CH^r(X_K;\Q),
\]
is surjective. Finally 
\[
\CH_{\hom}^r(S \times_{\overline\Q} X;\Q) \twoheadrightarrow  \CH_{\hom}^r(U\times_{\ol{\Q}}X;\Q),
\]
is surjective by the Hodge conjecture, and the proposition follows.}
\end{proof}

By our assumptions, 
$$\Phi _r :\ \CH^r_{\hom}((S\times X)_{\mathbb{\ol{Q}}};\mathbb{Q})\hookrightarrow J\big(H^{2r-1}(S\times X,\mathbb{Q}(r))\big).$$
We have the following decomposition at the level of jacobians.
$$J(H^{2r-1}(S\times X,\mathbb{{Q}}(r)))\cong J(H_0)\oplus J(H^{\perp }_0),$$
where $H^{\perp }_0$ arises due to polarization. Let $P_1$ be the projector
$$H^{2r-1}(S\times X , \mathbb {Q}(r))\twoheadrightarrow H_0,$$
and $w_1$ be an algebraic cycle {lying in the K\"unneth component 
\[
H^{2(d+\nu-r)-1}(S\times X,\Q(d+\nu-r))\otimes H^{2r-1}(S\times X,\Q(r))
\]}
 corresponding to it.  Let 
\[
\Xi _1:=w_{1,*}(\CH^r_{\hom}((S\times X)_{\overline{\mathbb{Q}}};\mathbb{Q})).
\]
 Since we are assuming Conjecture \ref{BBC}, $F^2CH^r((S\times X)_{\overline{\mathbb{Q}}};\mathbb{Q})=0$ and $\Xi _1$ is independent of the choice of algebraic cycle  representative corresponding to $P_1$. Viewing everything inside the jacobian, we get 
$$P_{1,*}: \CH^r_{\hom}((S\times X)_{\overline{\mathbb{Q}}};\mathbb{Q})\twoheadrightarrow Gr^{\nu }_F\CH^r(X_K;\mathbb{Q})$$
and
$$\Phi _r|_{\Xi _1} : \Xi _1\cong Gr^{\nu }_F\CH^r(X_K;\mathbb{Q}).$$
By a similar procedure, we get $\Xi _2:=w_{2,*}(\CH^{d-r+\nu }_{\hom}((S\times X)_{\overline{\mathbb{Q}}};\mathbb{Q}))$, for an algebraic cycle $w_2$ (similar to $w_1$), and an isomorphism
$$\Phi _{d-r+\nu }|_{\Xi _2} : \Xi _2\cong Gr^{\nu }_F\CH^{d-r+\nu }({X_K};\mathbb{Q}).$$
Note that $d-r+\nu =(d+\nu -1)-r+1$ and we have Beilinson's height pairing
$$\CH^r_{\hom}((S\times X)_{\overline{\mathbb{Q}}};\mathbb{Q})\times \CH^{d-r+\nu }_{\hom}((S\times X)_{\overline{\mathbb{Q}}};\mathbb{Q})\rightarrow \mathbb{R},$$
and hence between $\Xi _1$ and $\Xi _2$. The desired pairing {$\langle\hspace{0.1cm},\hspace{0.1cm}\rangle ^{\nu }_{\rm HT}$} between the spaces  $Gr^{\nu }_F\CH^r(X_K;\mathbb{Q})$ and $Gr^{\nu }_F\CH^{d-r+\nu }(X_K;\mathbb{Q})$ is now obtained through the isomorphisms above.
\end{proof}

\begin{rmk}
One can show that the height pairing above is independent of the choice of smooth 
projective variety $S/\overline{\mathbb{Q}}$ with $\overline{\mathbb{Q}}(S)\cong K$.
\end{rmk}

Since our height pairing {$\langle \hspace{0.1cm},\hspace{0.1cm}\rangle ^{\nu }_{\rm HT}$} is given by the one developed by Beilinson, it is only natural 
that the conjectures in \cite{Be3} have a natural extension for graded pieces. As an example:

\begin{prop}\label{HI}
Assume Conjecture \ref{Beconj2} (Hodge-index conjecture) and let $L_{X_K}$ denote the operation of 
intersecting with a hyperplane section. Then for $x\neq 0\in Gr^{\nu }_FCH^r(X_K;\mathbb {Q})$ 
such that $L^{d-2r+\nu +1}_{X_K}(x)=0$, the height pairing
$$(-1)^r\langle x\hspace{0.1cm},\hspace {0.1cm}L^{d-2r+\nu }_{X_K}(x)\hspace {0.1cm}\rangle ^{{\nu }}_{\rm HT}\hspace{0.1cm}>\hspace {0.1cm}0,$$
when $r\leq (d+\nu )/2$. 
\end{prop}
\begin{proof}
First note that the filtration developed in \cite{Lew2} already has the property that $L^{d-2r+\nu }_{X_K}$ 
defines an isomorphism between $Gr^{\nu }_F\CH^r(X_K;\mathbb{Q})$ and 
$Gr^{\nu }_F\CH^{d-r+\nu }(X_K;\mathbb{Q})$, so Proposition \ref{HI} makes sense. Now for any $x\in \Xi _1$
$$\Phi _{d-r+\nu }\left(L^{d-2r+\nu }_{S\times X}(x)-w_{2,*}\circ L^{d-2r+\nu }_{S\times X}(x)\right)$$
$$=[L_{S\times X}]^{d-2r+\nu }(\Phi _r(x))-[w_2]_*\circ [L_{S\times X}]^{d-2r+\nu }(\Phi _r(x))$$
$$=[L_{S\times X}]^{d-2r+\nu }(\Phi _r(x))-[L_{S\times X}]^{d-2r+\nu }(\Phi _r(x))=0.$$
Since we are assuming Conjecture \ref{BBC}, we get
$$L^{d-2r+\nu }_{S\times X}(x)=w_{2,*}\circ L^{d-2r+\nu }_{S\times X}(x),$$
which shows that $L^{d-2r+\nu }_{S\times X}$ maps $\Xi _1$ to $\Xi _2$, isomorphically. 
Further, let $\Xi '_{2}\subset \CH^{d-r+\nu +1}_{\hom}(S\times X;\mathbb{Q})$ be such 
that $\Xi '_{2}\cong Gr^{\nu }_F\CH^{d-r+\nu +1}(X_K;\mathbb{Q})$.\\
For $x'\in \Xi _1$, 

$\Phi _r(x')=x\in Gr^{\nu }_F\CH^r(X_K;\mathbb {Q})\implies \Phi _{d-r+\nu +1}(L^{d-2r+\nu +1}_{S\times X}(x'))=L^{d-2r+\nu +1}_{X_K}(x).$
So, $L^{d-2r+\nu +1}_{X_K}(x)=0\implies L^{d-2r+\nu +1}_{S\times X}(x')=0$. We also have
$$(-1)^r\langle x\hspace{0.1cm},\hspace {0.1cm}L^{d-2r+\nu }_{X_K}(x)\hspace {0.1cm}\rangle ^{{\nu }}_{\rm HT}=(-1)^r\langle x'\hspace{0.1cm},\hspace {0.1cm}L^{d-2r+\nu }_{S\times X}(x')\hspace {0.1cm}\rangle _{\rm HT}\hspace {0.1cm}.$$
Note that $x'\in \Xi _1\subset \CH^r_{\hom}((S\times X)_{\overline{\Q}};\mathbb {Q})$ and $L^{d-2r+\nu +1}_{S\times X}(x')=0$. Now assuming Conjecture \ref{Beconj2} , we conclude
$$(-1)^r\langle x'\hspace{0.1cm},\hspace {0.1cm}L^{d-2r+\nu }_{S\times X}(x')\hspace {0.1cm}\rangle _{\rm HT}\hspace {0.1cm}>\hspace {0.1cm}0\hspace {0.1cm},$$
and Proposition \ref{HI} follows immediately.
\end{proof}

We study the following subspace of $Gr^{\nu }_F\CH^r(X_K;\mathbb{Q})$:
\begin{defn} Let $F^{\nu }\underline {\CH}^r_{\alg}(X_{K};\mathbb {Q}) :=$
$$ F^{\nu }\CH^r(X_{K};\mathbb {Q})\bigcap \left[{\rm Im} (\CH^{r}_{\alg}((S\times X)_{\overline {\mathbb {Q}}};\mathbb {Q})\rightarrow  
\CH^r(X_{K};\mathbb {Q}))\right].$$
Then we define
$$Gr^{\nu }_{F}\underline {\CH}^r_{\alg}(X_{K};\mathbb {Q}) := {\rm Im}\left(F^{\nu }\underline {CH}^r_{alg}(X_{K};\mathbb {Q})\rightarrow Gr^{\nu }_{F}\CH^r(X_{K};\mathbb {Q})\right).$$
\end{defn}

There is one remark in order: If $S'$ is another such variety, then we can dominate both $S$ and $S'$ by a desingularization of
a third $S^{''}\hookrightarrow S\times S'$. From this, and the fact that the rational Chow group of cycles 
algebraically equivalent to zero being a $\mathbb {Q}$ vector space, one can show
$${\rm Im} \left(\CH^{r}_{\alg}((S\times X)_{\overline {\mathbb {Q}}};\mathbb {Q})\rightarrow \CH^r(X_{K};\mathbb {Q})\right)
$$
and
$$  {\rm Im} \left(\CH^{r}_{\alg}((S'\times X)_{\overline {\mathbb {Q}}};\mathbb {Q})\rightarrow CH^r(X_{K};\mathbb {Q})\right),$$
are the same.
Thus the definition of $Gr^{\nu }_{F}\underline {\CH}^r_{\alg}(X_{K};\mathbb {Q})$ is independent of the choice of $S$. Now we have the following

\begin{thm}\label{IIITh2}
Under the same set up as in Theorem \ref{IIITh1} , we have the height pairing
$$\langle\  ,\ \rangle ^{{\nu }}_{\rm HT, {alg}} : Gr^{\nu }_{F}\underline {\CH}^r_{\alg}(X_{K};\mathbb {Q})\times Gr^{\nu }_{F}\underline {\CH}^{d-r+\nu }_{\alg}(X_{K};\mathbb {Q})\rightarrow \mathbb{R} ,$$
extending the N\'eron-Tate pairing.
\end{thm}

\begin{proof}
Assuming Conjecture \ref{BBC}, we get that
$$\CH^r_{\alg}((S\times X)_{\overline {\mathbb {Q}}};\mathbb {Q})\hookrightarrow J^r_{\alg}(S\times X)_{\mathbb{Q}},$$
and
$$\CH^{d-r+\nu }_{\alg}((S\times X)_{\overline {\mathbb {Q}}};\mathbb {Q})\hookrightarrow J^{d-r+\nu }_{\alg}(S\times X)_{\mathbb{Q}}.$$

The proof now goes exactly in the same way as Theorem \ref{IIITh1}, if we replace $\CH^r_{\hom}((S\times X)_{\overline {\mathbb {Q}}};\mathbb {Q})$ 
(resp.  $\CH^{d-r+\nu }_{\hom}((S\times X)_{\overline {\mathbb {Q}}};\mathbb {Q})$) with $\CH^r_{\alg}((S\times X)_{\overline {\mathbb {Q}}};\mathbb {Q})$ (resp.  $\CH^{d-r+\nu }_{\alg}((S\times X)_{\overline {\mathbb {Q}}};\mathbb {Q})$). 
We obtain $\Xi _{1,\alg}\subset \CH^r_{\alg}((S\times X)_{\overline {\mathbb {Q}}};\mathbb {Q})$ (respectively $\Xi _{2,\alg}\subset \CH^{d-r+\nu }_{\alg}((S\times X)_{\overline {\mathbb {Q}}};\mathbb {Q})$), such that
$$\Xi _{1,\alg}\cong  Gr^{\nu }_{F}\underline {\CH}^r_{\alg}(X_{K};\mathbb {Q}), \quad
\Xi _{2,\alg}\cong  Gr^{\nu }_{F}\underline {\CH}^{d-r+\nu }_{\alg}(X_{K};\mathbb {Q}).$$
The height pairing $\langle\  ,\ \rangle ^{{\nu }}_{\rm HT, {alg}}$ is now given as the pairing between $\Xi _{1,\alg}$ and $\Xi _{2,\alg}$.
\end{proof}
{\begin{rmk}\label{7R8}
The reasons for restricting to this particular subspace are the following:
\begin{enumerate}
\item{Since the height pairing for cycles algebraically equivalent to zero is given by the N\'eron-Tate pairing (\cite{Be3}, Remark 4.0.8), one can work without the assumption of Conjecture \ref{BBC}.}
\vspace{0.5cm}
\item{Further for cycles algebraically equivalent to zero, assumption (17) of \cite{Ku} is no longer necessary (\cite{Ku}, \S 8).}
\end{enumerate}
\vspace{0.3cm}
So in effect, one can freely use the machineries available from arithmetic intersection theory to compute the height pairing, albeit (GHC). We will illustrate this with an example.
\end{rmk}}

\subsection{An example computation}
Using the formalism of arithmetic intersection theory {discussed} in \S\ref{S3}, we present here a computation related to the theory developed so far. 
\begin{ex}\label{EX5}
Let $X = C_{1}\times C_{2}$ be the product of smooth projective curves $C_1$ and $C_2$, defined over $\overline{\mathbb{Q}}$. For $\nu =2$, we fix an embedding $K = \overline {\mathbb {Q}}(C_{2})\hookrightarrow \mathbb {C}$ ({so naturally $S=C_2$ following the set up of Theorem \ref{IIITh1}}), and let $p = \eta _{2}\in C_{2}(\mathbb {C})$ be a very general point corresponding to this embedding. {To be more precise, $\eta_2$ is regarded  as the generic point
of $\ol{\Q}(C_2)$ (so $\ol{\Q}(C_2) = \ol{\Q}(\eta_2)$), and recall that any point $p\in S(\C)$ for which evaluation at $p$ defines an embedding
$\ol{\Q}(\eta_2)\hookrightarrow \C$ is defined to be a very general point. Although this notation  is a bit slang, we write
$p = \eta_2$.}
We fix $e_{2}\in C_{2}(\overline {\mathbb {Q}})$. For distinct points $p_{1},q_{1},p_{2},q_{2}\in C_{1}(\overline {\mathbb {Q}})$, let

$$\xi _{1}:=(p_{1}-q_{1})\times (\eta _{2}-e_{2})\in Gr^{2}_{F}\underline{\CH}^2_{\alg}(X_{K};\mathbb {Q}) ,$$
$$\xi _{2}:=(p_{2}-q_{2})\times (\eta _{2}-e_{2})\in Gr^{2}_{F}\underline{\CH}^{2}_{\alg}(X_{K};\mathbb {Q}).$$
Assume also
\begin{equation}\label{EQ5}
N^1_{H}\left (H^1(C_{1},\mathbb {Q})\otimes H^1(C_{2},\mathbb {Q})\right )= 0.
\end{equation}
Then
$$\langle \xi _{1}, \xi _{2} \rangle ^{{2}}_{\rm HT, {alg}} = deg \left(\Delta ^{2}_{C_{2}}(1,1)\right)\langle p_{1}-q_{1}, p_{2}-q_{2}\rangle _{\rm NT},$$
where on the RHS we have N\'eron-Tate pairing.
\end{ex}
We add a remark before we begin the proof:
\begin{rmk}
The assumption in (\ref{EQ5}) holds for example if we take $X=E_1\times E_2$, a product of two non-isogenous elliptic curves; for
here $N^1_H(H^1(E_1,\mathbb{Q})\otimes H^1(E_2,\mathbb{Q}))=H^2_{alg}(E_1\times E_2,\mathbb{Q})\cap (H^1(E_1,\mathbb{Q})\otimes H^1(E_2,\mathbb{Q}))=0$
follows from the fact that any non-zero element
$$[\xi ]\in H^2_{\alg}(E_1\times E_2,\mathbb{Q})\cap (H^1(E_1,\mathbb{Q})\otimes H^1(E_2,\mathbb{Q}))$$
will in turn define an isogeny between $E_1$ and $E_2$.
\end{rmk}
Since Example \ref{EX5} illustrates the potential of our theory so effectively, we will provide a more detailed proof.
\begin{proof}
{From the assumptions of Example \ref{EX5}} we get that $S\times X=C_2\times (C_1\times C_2)\cong C_1\times (C_2\times C_2)$. 
We have the Chow-K\"unneth decomposition for smooth curves
$$\Delta _{C_{2}}(1,1) = \Delta _{C_{2}}-e_{2}\times C_{2}- C_{2}\times e_{2} .$$
Now put
$$\tilde{\xi }_1:=(p_1-q_1)\times \Delta _{C_2}(1,1)\in \CH^2_{\alg}(C_1\times (C_2\times C_2);\mathbb{Q}),$$
$$\tilde{\xi }_2:=(p_2-q_2)\times \Delta _{C_2}(1,1)\in \CH^2_{\alg}(C_1\times (C_2\times C_2);\mathbb{Q}).$$
{Using the assumption in  (\ref{EQ5}) and basic intersection theory, one can show $\tilde{\xi }_i\in \Xi _{alg}\mapsto \xi _i, i=1,2$ under the isomorphism $\Xi _{alg}\cong Gr^2_{F}\underline{\CH^2}_{alg}(X_K;\Q)$. Here $\Xi _{alg}\subset \CH^2_{\alg}(C_1\times (C_2\times C_2);\mathbb{Q})$ is the suitable subspace (see Theorem \ref{IIITh2} for details)}. Thus, the height pairing {$\langle \xi _{1}, \xi _{2} \rangle ^{{2}}_{\rm HT, {alg}}$} is given by $\langle \tilde{\xi }_1,\tilde{\xi }_2\rangle_{\rm HT}$. We provide a general computation in the formalism of arithmetic intersection theory, 
for which $\langle \tilde{\xi }_1,\tilde{\xi }_2\rangle_{\rm HT}$ is a particular case.
\begin{lemma}\label{IIILe1}
Let $C$ be smooth projective curve and $X$ be a smooth projective variety of dimension $d-1$, both
defined over a number field $k$. Let $\alpha _{1} ,\alpha _{2}  \in \CH^1_{\alg}(C;\mathbb {Q})$ and $\pi _{1}: C\times X\rightarrow C$ and $\pi _{2} : C\times X\rightarrow X$ are the projections. Given $w_{1}\in \CH^{r-1}(X;\mathbb {Q})$ and $w_{2}\in 
\CH^{(d-1)-(r-1)}(X;\mathbb {Q})= \CH^{d-r}(X;\mathbb {Q})$ and the cycles
$$a _{1} :=\pi ^{*}_{1}(\alpha _{1})\cdot \pi ^*_{2}(w_{1})\in \CH^r_{\alg}(C\times X;\mathbb {Q})$$

$$a_{2} := \pi ^{*}_{1}(\alpha _{2})\cdot \pi ^*_{2}(w_{2})\in \CH^{d-r+1}_{\alg}(C\times X;\mathbb {Q}).$$
We get the following height pairing relation :
$$\langle a_{1},a_{2}\rangle _{\rm HT} = (\deg(w_{1}\cdot w_{2})_X)\langle \alpha _{1},\alpha _{2}\rangle _{\rm NT}\hspace {0.1cm},$$
where $(w_{1}\cdot w_{2})_X$ is the usual intersection pairing on $X$.
\end{lemma}
\begin{proof}
Let $\widetilde {C}$ be a regular semi-stable model for $C$ over {$\Spec(\mathcal{O}_{k'})$} (after a finite extension $k'$ of the ground field $k$). Choose $Z_i,\hspace {0.1cm}i=1,2$ cycles on $\widetilde {C}$ of codimension $1$ such that
\begin{enumerate}
\item{$Z_i|_{C} =\alpha _i$.}
\item{$Z_i\cdot V=0$ for any vertical cycle $V$.}
\end{enumerate} 
One can arrange the above situation by Th. 1.3 of \cite{Hr}. Choose $g_i,\hspace {0.1cm}i=1,2$, Green's functions for $Z_i$ such that $dd^cg_i+\delta _{{Z _i}}=0$ ({since $\alpha _i$ is null-homologous, the cohomology class $[\omega _{Z_i}]=0$}). We have
$$[(Z_i,g_i)]\in \widehat {\CH}^1(\widetilde {C}),\hspace {0.1cm}i=1,2\hspace {0.1cm}.$$
Then,
$$\langle \alpha _1,\alpha _2\rangle _{\rm NT} = \widehat{deg}_{\widetilde {C}}\left([(Z_1,g_1)]\cdot [(Z_2,g_2)]\right)\in \widehat {\CH}^1(\Spec(\mathbb {Z}))\otimes \Q\cong \mathbb {R},$$
is independent of the choices of $Z_i,g_i$.\\

Now, for any projective and flat model $\widetilde{X}'$ over {$\Spec(\mathcal{O}_{k'})$} of $X$, by de Jong's alteration (\cite {dJo}, Theorem 8.2) we get a projective, flat and regular scheme $\widetilde {X}$ {over a finite extension of $k'$ (in turn a finite extension of $k$)}, with a finite and surjective morphism to $\widetilde{X}'$. In particular $\dim(\widetilde{X}')=\dim(\widetilde {X})$. Let $W_i,\hspace {0.1cm}i=1,2$ be cycles on $\widetilde {X}$ of codimensions $r-1$ and $d-r$ respectively such that 
$$W_i|_{X}=w_i,\hspace {0.2cm}i=1,2\hspace {0.1cm}.$$
Let $g_{W_{1}}$ (resp. $g_{W_2}$) be {a Green current} for $W_1$ (resp. $W_2$). Then 
$$[(W_1,g_{W_1})]\in \widehat{\CH}^{r-1}(\widetilde {X})$$
$$[(W_2,g_{W_2})]\in \widehat{\CH}^{d-r}(\widetilde {X})\hspace {0.1cm}.$$
For the scheme {$\widetilde {C}\times _{\Spec(\mathcal{O}_{k'})}\widetilde {X}$, we can use the alteration trick once more to obtain a regular flat and projective scheme {$Z$ over $\Spec(\mathcal{O}_{k^{''}})$, where $k^{''}$ is a finite extension of $k$} and a dominant and finite morphism $f : Z\rightarrow \widetilde {C}\times _{\Spec(\mathcal{O}_{k'})}\widetilde {X}$. In particular $\dim(Z)=\dim (\widetilde {C}\times _{\Spec(\mathcal{O}_{k'})}\widetilde {X})=d+1$}. For the projections
$$\pi _{\widetilde {C}} : \widetilde {C}\times _{\Spec(\mathcal{O}_{k'})}\widetilde {X}\rightarrow \widetilde {C}$$
$$\pi _{\widetilde {X}} : \widetilde {C}\times _{\Spec(\mathcal{O}_{k'})}\widetilde {X}\rightarrow \widetilde {X}\hspace {0.1cm},$$
consider
$$f_{\widetilde {C}} := \pi _{\widetilde {C}}\circ f$$
$$f_{\widetilde {X}} := \pi _{\widetilde {X}}\circ f\hspace {0.1cm}\hspace {0.1cm},$$
and the cycles
$$\tilde {a}_1:=f^*_{\widetilde {C}}([(Z_{1},g_{Z_{1}})]){\cdot }f ^*_{\widetilde {X}}([(W_{1},g_{W_{1}})])$$
$$\tilde {a}_2:=f^*_{\widetilde {C}}([(Z_{2},g_{Z_{2}})]){\cdot }f^*_{\widetilde {X}}([(W_{2},g_{W_{2}})])\hspace {0.1cm}.$$
Then ({up to rational multiples, which will arise since we are using alterations and extensions of the base field $k$})
$$\langle a_1,a_2\rangle _{\rm HT} = \widehat{deg}_{Z}\left(\tilde {a}_1\cdot \tilde {a}_2\right)\in \widehat {\CH}^1(\Spec(\mathbb{Z}))\otimes \mathbb{Q}\cong \mathbb {R}\hspace {0.1cm}.$$
Since $f^*_{\widetilde {C}}$ and $f^*_{\widetilde {X}}$ are morphisms of rings (\cite {GS}, 4.4.3 (5)),
$$\tilde {a}_1\cdot \tilde {a}_2 = f^*_{\widetilde {C}}\left([(Z_1,g_1)]\cdot [(Z_2,g_2)]\right)\cdot f^*_{\widetilde {X}}\left([(W_1,g_{W_1})]\cdot [(W_2,g_{W_2})]\right)\hspace {0.1cm}.$$
By the projection formula for arithmetic intersection pairing (\cite {GS}, 4.4.3 (7))
$$f_{\widetilde {C},*}\left(\tilde {a}_1\cdot \tilde {a}_2\right) = \underbrace {[(Z_1,g_1)]\cdot [(Z_2,g_2)]}_{\in \widehat {\CH}^2(\widetilde {C};\Q)}\cdot \underbrace {f_{\widetilde {C},*}\left[f^*_{\widetilde {X}}\left([(W_1,g_{W_1})]\cdot [(W_2,g_{W_2})]\right)\right]}_{\in 
\widehat {\CH}^{0}(\widetilde {C};\Q))}.$$
Since
$$\widehat{deg}_{Z}\left(\tilde {a}_1\cdot \tilde {a}_2\right) = \widehat{deg}_{\widetilde {C}}\left(f_{\widetilde {C},*}(\tilde {a}_1\cdot \tilde {a}_2)\right)$$
and
$$f_{\widetilde {C},*}\left[f^*_{\widetilde {X}}\left([(W_1,g_{W_1})]\cdot [(W_2,g_{W_2})]\right)\right]=\deg(w_1\cdot w_2)_X,$$
we obtain our desired result. {We note here that since we are using $\Q$-valued intersection pairing, the relations among various height-pairings won't change.}
\end{proof}
Quite generally, one can also prove the following:

\begin{thm}[\cite{S-G}]\label{8C1}
Given smooth projective curves ${C_{1},\dots ,C_{d}}$ over $\overline {\mathbb {Q}}$,  let $X = C_{1}\times \cdots \times C_{d}$. For $\nu \geqq 2$, we fix an embedding $K = \overline {\mathbb {Q}}(C_{2}\times \cdots \times C_{\nu })\hookrightarrow \mathbb {C}$, and let $p = (\eta _{2},\cdots ,\eta _{\nu })\in C_{2}(\mathbb {C})\times \cdots \times C_{\nu }(\mathbb {C})$ be a very general point corresponding to this embedding {(see Example
\ref{EX5} for a clarification of ``general'').} We fix $e_{j}\in C_{j}(\overline {\mathbb {Q}})$, $j=2,\cdots ,d$. For distinct points $p_{1},q_{1},p_{2},q_{2}\in C_{1}(\overline {\mathbb {Q}})$ and $\nu \leqq r\leqq d$, let
\[
\xi _{1}:=Pr^{*}_{1,\cdots ,\nu }((p_{1}-q_{1})\times (\eta _{2}-e_{2})\times \cdots \times (\eta _{\nu }-e_{\nu}))\bigcap 
\]
\[
Pr^{*}_{\nu +1,\cdots ,r}(e_{\nu +1},\cdots ,e_{r})\in Gr^{\nu }_{F}{\underline{\CH}^r_{alg}(X_{K};\mathbb {Q})},
\]
\[
\xi _{2}:=Pr^{*}_{1,\cdots ,\nu }((p_{2}-q_{2})\times (\eta _{2}-e_{2})\times \cdots \times (\eta _{\nu }-e_{\nu }))\bigcap
\]
\[
 Pr^{*}_{r+1,\cdots ,d}(e_{r+1},\cdots ,e_{d})\in {Gr^{\nu }_{F}\underline{\CH}^{d-r+\nu }_{alg}(X_{K};\mathbb {Q})}.
\]
Assume also 
\[
N^1_{H}\left (H^1(C_{1},\mathbb {Q})\otimes \cdots \otimes H^1(C_{\nu },\mathbb {Q})\right)  = 0,
\]
and
\[
N^1_{\overline{\Q}}\left (H^1(C_{2},\mathbb {Q})\otimes \cdots \otimes H^1(C_{\nu },\mathbb {Q})\right)=0 .
\]
Then, ${\langle \xi _{1}, \xi _{2} \rangle ^{\nu }_{\rm HT, alg}} =  \left(\prod^{\nu }_{j=2} [deg (\Delta ^{2}_{C_{j}}(1,1))_{C_{j}\times C_{j}}] \right)\langle p_{1}-q_{1}, p_{2}-q_{2}\rangle _{\rm NT}$,\\
where $\langle\ ,\  \rangle _{\rm NT}$ is the N\'eron-Tate pairing on $(J^{1}(C_{1})(\overline {\mathbb {Q}}))\otimes \mathbb {Q}$.
\end{thm}
\end{proof}
\begin{rmk}
For a self-product of a CM-elliptic curve (\cite{S-G}, \S 8.2), we were able to eliminate the  assumption 
in (\ref{EQ5}) altogether.
\end{rmk}

\section{An Archimedean pairing involving the equivalence relation defining higher Chow groups}\label{S6}

In this section, and for each $m\geq 0$, we construct a pairing
on the cycle level, involving the equivalence relation
in the definition of Bloch's higher Chow groups $\CH^r(X,m)$ defined below.
The case when $m=0$ has already been defined in \S\ref{S1}, and
the nature of this pairing is more akin to the Archimedean height
pairing defined in the literature. It was first discussed in \cite{C-L}; however
presentation here is intended to be more user friendly.
A general construction of this pairing for all
$m$ is in order. We first recall that two subvarieties $V_1,\ V_2$
of a given variety  intersect properly if codim$\{V_1\cap V_2\} \geq$ codim $V_1 \ +$
codim $V_2$. This notion naturally extends to algebraic cycles.

\medskip
\noindent
(i) {\it Higher Chow groups.}\ Let $W/k$ be a 
quasi-projective variety over a field $k$. Put
$z^r(W) =$ free abelian group generated by subvarieties of
codimension $r$ in $W$, 
\[
\Delta^m := \Spec\biggl(\frac{k[t_0,...,t_m]}{1-\sum_{j=0}^mt_j}\biggr)
\]
 the standard $m$-simplex, and
$z^r(W,m) = \big\{\xi\in z^k(W\times \Delta^m)\ \big|$ $\xi$ meets
all faces $\{j_1=\cdots=j_\ell=0\ |\  \ell=1,...,m\}$ properly$\big\}$.

\begin{defn} [\cite{Blo1}]
${\CH}^{\bullet}(W,\bullet) =$ homology of $\big\{z^{\bullet}
(W,\bullet),\partial\big\}$.
 We put ${\rm CH}^k(W) := {\rm CH}^k(W,0)$. 
 \end{defn}
 
\noindent
(ii) {\it Cubical version.}\ Let $\square^m:= ( \PP^1 \backslash
\{1\})^m $ with coordinates $z_i$ and $2^{m}$ codimension one faces
obtained by setting $z_i=0,\infty$, and boundary maps
$\partial = \sum (-1)^{i-1}(\partial_{i}^{0}-\partial_{i}^{\infty})$,
where $\partial_{i}^{0},\ \partial_{i}^{\infty}$ denote
the restriction maps to the faces $z_{i}=0,\ z_{i}=\infty$ 
respectively. The rest of the definition is completely
analogous for $z^r(W,m) \subset z^r(W\times \square^m)$,
except that one has to quotient out by the
subgroup $z^r_{\dgt}(W,m) \subset z^r(W,m)$ of degenerate cycles
obtained via pullbacks $\sum_{j=1}^m\pr_j^{\ast} : z^r(W,m-1) \to z^r(W,m)$,
$\pr_j : W\times \square^m \to W\times \square^{m-1}$ the
$j$-th canonical projection. 
It is known that both complexes are quasi-isomorphic (Bloch (unpublished)/Levine \cite{Lv}; independently).

\subsection{A quick detour via Milnor $K$-theory}  
An excellent reference for this part is \cite{B-T}.
Let $\FF$ be a field with multiplicative group $\FF^{\times}\subset \FF$. 
Consider the graded  tensor algebra
\[
T(\FF) := \bigoplus_{r=0}^{\infty}\{\FF^{\times}\}^{\otimes_{\Z}m} = \Z \oplus \FF^{\times}\oplus \cdots,
\]
and let $R(\FF)$ be the graded $2$-sided ideal generated by
\[
\big\{\tau \otimes (1-\tau)\ \big|\ \ \tau \in \FF^{\times}\bs \{1\}\big\}.
\]
Recall that the Milnor $K$-theory of $\FF$ is given by
\[
K^M_{\bullet}(\FF)  := T(\FF)/R(\FF) = \bigoplus_{r=0}^{\infty}K_m^M(\FF).
\]
Further, recall that
 $K_r^M(\FF) \simeq \CH^r(\Spec(\FF),r)$, (Nesterenko/Suslin (1990), Totaro (1992)).
 Now let $W/k$ be a smooth scheme over a field $k$.
 If one replaces $\FF$ by $\CO_W^{\times}$, then we arrive at the sheaf $\K^M_{r,W}$ of 
 Milnor $K$-groups.  {To be more precise,
 let ${\CO}_{W}$ be the sheaf of regular functions on $X$,
with sheaf of units ${\CO}_{W}^{\times}$. As in \cite{Ka}, we
put
$$
{\K}_{r,W}^{M} :=  \big(\underbrace{{\CO}_{W}^{\times}\otimes 
\cdots\otimes {\CO}_{W}^{\times}}_{r\ {\rm times}}\big)\big/
{\J},\quad
{\rm (Milnor \ sheaf)},
$$
where $\J$ is the subsheaf of the tensor product
generated by sections of the form:
$$
\big\{\tau_{1}\otimes\cdots\otimes \tau_{r}\ \big|\
\tau_{i} +\tau_{j} = 1,\quad{\rm for\ some}\ i\ {\rm and}\ j,\ i\ne j
\big\}.
$$
For example, ${\K}_{1,W}^{M} = {\CO}^{\times}_{X}$.}
The higher Chow
groups $\CH^r(W,m)$ come naturally equipped with
a coniveau filtration involving codimension of cycles when projected to $W$, whose 
graded pieces can be computed via a local-to-global
spectral sequence (\cite{BO}, \cite{Blo1}), involving flasque resolutions of certain sheaves. 
Via the works of Elbaz-Vincent$/$M\"uller-Stach (1998), and Gabber (1992),
 (see \cite{MS2}, together with \cite{Ke}), one of those sheaves is
$\K^M_{r,W}$. This, together with partial degeneration of the aforementioned spectral sequence leads to:

\begin{thm}[See \cite{MS2}]  For $0\leq m\leq 2$, there is an isomorphism
\[
H^{r-m}_{\Zar}(W,\K^M_{r,W}) \simeq \CH^r(W,m).
\]
\end{thm}
In the context of Milnor $K$-theory, the last 3 terms of the flasque resolution of $\K^M_{r,W}$ are
\[
 \bigoplus_{{\rm cd}_WV=r-2}K^M_2(\C(V))
\xrightarrow{\rm Tame}
\bigoplus_{{\rm cd}_WV=r-1}K_1^M(\C(V)) \xrightarrow{\Div} \bigoplus_{{\rm cd}_WV=r}K_0^M(\C(V)).
\]

If we interpret this in terms of global sections, this leads to a complex
whose last three terms and corresponding homologies (norm/graph maps, indicated at $\updownarrow$)
for $0\leq m\leq 2$ are:
\begin{equation}\label{E345}
\begin{matrix} \bigoplus_{\text{\rm cd}_WZ=r-2}
K^M_2({\C}(Z))&{\buildrel T\over \to}&\bigoplus_{\text{\rm cd}_WZ=r-1}
{\C}(Z)^\times&{\buildrel \text{\rm div}\over\to}& 
\bigoplus_{\text{\rm cd}_WZ=r} {\Z}\\
&\\
\updownarrow&&\updownarrow&&\updownarrow\\
&\cr
{\CH}^{r}(W,2)&&{\CH}^{r}(W,1)&&{\CH}^{r}(W,0)\end{matrix}
\end{equation}
where as a reminder,  div is the divisor map of zeros minus poles
of a rational function, and $T$ is the Tame symbol map.
Again as a reminder, the  Tame symbol map
$$
T:\bigoplus_{\text{\rm cd}_XZ = r-2}{K}_{2}^M({\C}(Z))\to
\bigoplus_{\text{\rm cd}_XD = r-1}{K}_{1}^M({\C}(D)),
$$
is defined as follows. First ${K}_{2}^M({\C}(Z))$ is
generated by symbols $\{f,g\}$, $f,\ g\in {\C}(Z)^{\times}$,
under $f\otimes g \mapsto \{f,g\}$.

\bigskip
For $f,\ g \in {\C}(Z)^{\times}$, 
$$
T\big(\{f,g\}\big) = \sum_D(-1)^{\nu_D(f)\nu_D(g)}\biggl({f^{\nu_D(g)}
\over g^{\nu_D(f)}}\biggr)_D,
$$
where $\big(\cdots\big)_D$ means restriction to
the generic point of $D$, and $\nu_D$ represents order of a 
zero or pole along an irreducible divisor $D\subset Z$.

\begin{ex}{\rm  Taking cohomologies of the complex in (\ref{E345}), we
have:

\bigskip
(i)  $\CH^{r}(W,0) = z^r(W)/z^r_{\rat}(W) =: \CH^r(W)$.

\medskip
(ii) ${\CH}^{r}(W,1)$ is represented by classes of the form
$\xi = \sum_{j}(f_{j},D_{j})$, where codim$_{X}D_{j}=r-1$,
$f_{j}\in {\C}(D_{j})^{\times}$, and $\sum \text{\rm div}(f_{j}) = 0$;
 modulo the image of the Tame symbol.
\medskip

(iii) ${\CH}^{r}(W,2)$ is represented by classes in the kernel
of the Tame symbol; modulo the image of a higher Tame symbol.}
\end{ex}

\bigskip

In this section we will adopt the cubical version of $\CH^{\bullet}(W,\bullet)$,
albeit a simplicial version can also be arranged \cite{KLL}.
The intersection product for cycles in the cubical version, is easy
to define. On the level of cycles, and in $W\times W\times \square^{m+n}$,
one has
\[ 
z^r(W,m) \times z^k(W,n) \to z^{r+k}(W\times W,m+n);
\]
however the pullback along the diagonal
\[
z^{r+k}(W\times W,m+n) \to  z^{r+k}(W,m+n),
\]
is not well-defined, even for smooth $W$. In
particular, for smooth $W$, the issue of
when an intersection product is
defined, which is a general position statement
involving proper intersections, has to be addressed since we will
be working on the level of cycles.
On the level of
Chow groups, a moving lemma of Bloch (adapted to the cubical
situation) guarantees a pullback for smooth $W$:
\[
\CH^{\bullet}(W\times W,\bullet) \to \CH^{\bullet}(W,\bullet),
\]
and hence an intersection product for smooth $W$. 

\bigskip
Let us return to the situation where $X$ be a projective algebraic manifold of dimension $d$, and let
$z^r_{\rat}(X,m) := \partial\big(z^r(X,m+1)\big)\subset z^r(X,m)$ be the 
equivalence relation subgroup defining the higher
Chow groups $\CH^r(X,m)$. As in \cite{KLM}, we will need to restrict ourselves to
those precycles\footnote{As a reminder, for $\xi\in z^r(X,\bullet)$ to be a cycle, 
we require $\del \xi=0$.}  $z^r(X,\bullet)$  that are in general position with respect to the real 
subsets $X\times [-\infty,0]^{\bullet}$, and we will denote this by
$z^{\bullet}_{\R}(X,\bullet)$.
Now introduce
\[
\Lambda^0(r,m,X) = 
\]
\[
\biggl\{(\xi_1,\xi_2)\in z^r_{\R,\rat}(X,m) \times z^{d-r+m+1}_{\R,\rat}(X,m)\ \biggl|\
\begin{matrix} |\xi_1|\cap |\xi_2| = \emptyset\\
{\rm in}\ X\times \square^m\end{matrix} \biggr\},
\]
\[
\Lambda^+(r,m,X) = \biggl\{ (\xi_1,\xi_2) \in \Lambda^0(r,m,X) \ \biggl|\
\begin{matrix} \xi_1 = \partial\xi_1', \ \text{\rm where}\\
 \xi_1'\cap \xi_2\ \text{\rm is\ defined}\\
\text{\rm in}\ z_{\R}^{d+m+1}(X,2m+1)\end{matrix}\biggr\},
\]
\[
\Lambda^-(r,m,X) = \biggl\{ (\xi_1,\xi_2) \in \Lambda^0(r,m,X) \ \biggl|\
\begin{matrix} \xi_2 = \partial\xi_2', \ \text{\rm where}\\
 \xi_1\cap \xi_2'\ \text{\rm is\ defined}\\
\text{\rm in}\ z_{\R}^{d+m+1}(X,2m+1)\end{matrix}\biggr\},
\]
\[
\Lambda(r,m,X) \subset \Lambda^+(r,m,X) \bigcap \Lambda^-(r,m,X),
\]
is characterized by the requirement that $\xi_1'\cap\xi_2'$ is defined,
viz,
\[
\xi_1'\cap\xi_2' \in z_{\R}^{d+m+1}(X,2m+2).
\]

\begin{thm}\label{MTA}
There are natural pairings
\[
\langle \ ,\ \rangle^+_m :  \Lambda^+(r,m,X) \to \C/\Z(1),
\]
\[
\langle \ ,\ \rangle^-_m :  \Lambda^-(r,m,X) \to \C/\Z(1),
\]
which satisfy the following:
\medskip

{\rm (i)} (Reciprocity) On $\Lambda(r,m,X)$, $\langle\  ,\ \rangle^+_m = (-1)^m\langle\ ,\ \rangle_m^-$.

\medskip

{\rm (ii)} (Bilinearity) If $(\xi_1^{(1)},\xi_2),\ (\xi_1^{(2)},\xi_2)
 \in \Lambda^+(r,m,X)$, then
\[
\langle \xi_1^{(1)}+\xi_1^{(2)},\xi_2\rangle^+_m = 
\langle \xi_1^{(1)},\xi_2\rangle^+_m +
\langle \xi_1^{(2)},\xi_2\rangle^+_m.
\]
If $(\xi_1,\xi_2^{(1)}),\ (\xi_1,\xi_2^{(2)})\in \Lambda^-(r,m,X)$, then
\[
\langle\xi_1,\xi_2^{(1)}+\xi_2^{(2)}\rangle^-_m = 
\langle \xi_1,\xi_2^{(1)}\rangle^-_m +
\langle \xi_1,\xi_2^{(2)}\rangle^-_m.
\]

{\rm (iii)} (Projection formula)  Let $\pi: X\to Y$ be a flat surjective morphism between two smooth
projective varieties, with $\dim X=d$. Then
$\langle \xi_1, \pi^* \xi_2\rangle_m^{\pm} = \langle \pi_* \xi_1,
\xi_2\rangle_m^{\pm}$ for all $\xi_1\in z_{\rat}^r(X,m)$ and
$\xi_2 \in z_{\rat}^{d-r+m+1}(Y,m)$ with $(\pi_* \xi_1,
\xi_2)\in \Lambda^{\pm}(r+s-d,m,Y)$, where $s := \dim Y$.
\end{thm}

\begin{proof} We first recall the definition of Deligne
cohomology. Good sources for this are \cite{Ki}, \cite{Ja1} and \cite{KLM}. Let $\D_X^{\bullet}$ be the (fine) sheaf of 
complex-valued currents acting on $C^{\infty}$ complex-valued
compactly supported $(2d-\bullet)$-forms, where we recall $\dim X=d$. One has 
a decomposition into Hodge type:
\[
\D_X^{\bullet} = \bigoplus_{p+q=\bullet}\D_X^{p,q},
\]
where $\D_X^{p,q}$ acts on $(d-p,d-q)$ forms, with Hodge
filtration,
\[
F^r\D_X^{\bullet} = \bigoplus_{p+q=\bullet,p\geq r}\D_X^{p,q}.
\]
Likewise, for a subring $\bA\subseteq \C$, there is the (soft)\footnote{In the end, acyclicity is all that matters here.} sheaf subcomplex $\CC^{\bullet}_X(\bA) \subset \D_X^{\bullet}$
of $\bA$-coefficient Borel-Moore chains on $X$. The
global sections of a given sheaf ${\mathcal S}$ over $X$ will be
denoted by ${\mathcal S}(X)$.  Next, for a morphism
of complexes $\lambda : A^{\bullet} \to C^{\bullet}$, we recall the cone complex:
\[
\text{\rm Cone}(A^{\bullet} {\buildrel \lambda\over\longrightarrow} B^{\bullet}) = A^{\bullet}[1]
\oplus B^{\bullet},
\]
with differential
\[
\delta_D : A^{q+1}\oplus B^q \to A^{q+2}\oplus B^{q+1},
\quad (a,b) \ {\buildrel \delta_D\over \mapsto} \ (-da,\lambda(a) + db).
\]

\begin{defn} Fix a subring $\bA\subseteq \R$. The Deligne cohomology of $X$ is
given by
\[
H^i_{\D}(X,\bA(j)) := 
\]
\[
H^i\big(\text{\rm Cone}\big(\CC^{\bullet}_{X}(X,\bA(j)) \bigoplus F^j\D_X^{\bullet}(X) 
\xrightarrow{\varepsilon-l}\D^{\bullet}_{X}(X)\big)[-1]\big).
\]
\end{defn}
It is customary of thinking of currents as associated to homology.
Note that by simply regarding $\CC^{\bullet}_{X}(\bA(j))$, $F^j\D_X^{\bullet}$ as acyclic resolutions of
the respective sheaves $\bA(r)$ and $\Omega^j_{X,{\rm d}-{\rm closed}}$,  with
quasi-isomomorphisms, $\{\bA(j) \to 0\to \cdots\} \approx \CC^{\bullet}_{X}(\bA(j))$,
$\Omega_X^{\bullet\geq j} \approx F^j\D_X^{\bullet}$,
the above definition, when compared with the one in \cite{EV},
already incorporates Poincar\'e duality.

\begin{rmk} Generally speaking, one thinks of currents as well behaved under proper push-forwards, albeit with
no defined pull-back. However, the rules can be broken here if one replaces the sheaf complex of currents
on a given manifold with another which is quasi-isomorphic and having better properties with respect to pull-backs and
multiplication. The situation is well documented in \cite{Ki}(\S4) and \cite{K-L}(\S8). The reader should keep
this in mind in the discussion below. To simplify our notation, we will use the notation ``$\cdot$'' to refer to
multiplication of currents. Also we use the principal branch of the $\log$ function below.
\end{rmk}

Continuing with the proof of Theorem \ref{MTA},
we now recall the description of the  regulator on the level of complexes \cite{KLM}.
\[
\cl_{r,m,X} : \CH^r(X,m) \to H_{\D}^{2r-m}(X,\Z(r)),\ {\rm viz.,}\ \bA = \Z.
\]
Consider $\square^m$ with affine coordinates $(z_1,...,z_m)$
and introduce the currents: ($\delta_V$ means integration over $V$)
\[
\Omega_m := \biggl( \bigwedge_{j=1}^md\log z_j\biggr)\cdot \delta_{\square^m},
\]
\[
 T_{z_1} = \delta_{[-\infty,0]\times\square^{m-1}},...,T_m := T_{z_1}\cap\cdots\cap T_{z_m} = \delta_{[-\infty,0]^m} :=\int_{[-\infty,0]^m}(-),
\]
\[
R_m := \log z_1d\log z_2\wedge\cdots d\log z_m\cdot \delta_{\square^m}
\]
\[
- (2\pi\I)\log z_2 d\log z_3\wedge\cdots\wedge d\log z_m \cdot T_{z_1}
+\cdots
\]
\[
+ \ (-1)^{m-1}(2\pi\I)^{m-1}\log z_m \cdot {T_{z_1}\cap\cdots\cap T_{z_{m-1}}}.
\]
 For  $\xi\in z_{\R}^r(X,m)$, and if we let ${\rm pr}_{\square}: X\time\square^m\to \square^m$, ${\rm pr}_X: X\times\square^m\to X$ 
 be  the obvious projections, 
we consider the currents on $X$:
\[
T_m(\xi) := \int_{\xi}\pr^{\ast}_{\square}(T_m)\wedge \pr_X^{\ast}(-),
\]
\[
\Omega_m(\xi) := \int_{\xi}\pr^{\ast}_{\square}(\Omega_m)\wedge \pr_X^{\ast}(-),
\]
\[
R_m(\xi) := \int_{\xi}\pr^{\ast}_{\square}(R_m)\wedge \pr_X^{\ast}(-).
\]
One has the following identities \cite{KLM}:
\begin{equation}\label{E111}
\partial T_{m}(\xi) = T_{m-1}(\partial \xi), \ d[\Omega_m(\xi)] = 2\pi\I\Omega_{m-1}(\partial\xi),
\end{equation}
\[
d[R_m(\xi)] = \Omega_m(\xi) - (2\pi\I)^m{T_m(\xi)} - 2\pi\I R_{m-1}(\partial\xi).
\]

The map $\cl_{r,m,X}$ is induced (up to the normalizing twist $(2\pi\I)^{r-m}$) by
\begin{equation}\label{E777}
\xi\in z^r(X,m) \mapsto \big((2\pi\I)^mT_m(\xi),\Omega_m(\xi),R_m(\xi)\big),
\end{equation}
with the following caveat. One expects a quasi-isomorphism $z^r_{\R}(X,\bullet) \approx z^r(X,\bullet)$,
which certainly holds after tensoring with $\Q$ \cite{K-L}. Having said this,
by the very definition of $\Lambda^{\pm}$, we can drop the $\Q$-coefficients from this
discussion without compromising the theorem.
It is easy to check that 
\[
\big((2\pi\I)^mT_m(\xi),\Omega_m(\xi),R_m(\xi)\big) = (0,0,0) \ \text{\rm for}\ \xi\in z_{\dgt}^r(X,m).
\]
For $m=0$, note that $ (T_0(\xi),\Omega_0(\xi),(2\pi\I)^m R_0(\xi)) = (\xi,\delta_{\xi},0)$.
First an observation. For precycles $\alpha\in z^p(X,\ell)$ and $\beta\in
z^q(X,n)$ (in general position), one has
the relation \cite{KLM}:
\begin{equation}\label{SSE}
R_{\ell +n}(\alpha\cup \beta) = (-2\pi\I)^{\ell}{T_{\ell}}(\alpha) \cdot R_n(\beta) 
+ R_{\ell}(\alpha) \cdot \Omega_n(\beta).
\end{equation}

\subsection{The pairings}\label{SS008}
For $(\xi_1,\xi_2)\in \Lambda^+(r,m,X)$, we put
\begin{equation}\label{E666}
\langle\xi_1,\xi_2\rangle_m^+:= (-2\pi\I)^{m+1}{T_{m+1}(\xi_1{'})} \cdot R_m(\xi_2)  +
R_{m+1}(\xi_1{'}) \cdot \Omega_m(\xi_2)
 \end{equation}
 \[
 \in  \C/\Z(2m+1) \simeq \C/\Z(1),
\]
where the latter $\simeq$ is given by multiplication by $(-4\pi^2)^{-m}$,
and for $(\xi_1,\xi_2)\in \Lambda^-(r,m,X)$, we put (under  $\C/\Z(2m+1) \simeq \C/\Z(1)$),
\[
\langle\xi_1,\xi_2\rangle_m^- := (-2\pi\I)^{m}{T_{m}(\xi_1)} \cdot R_{m+1}(\xi_2') 
+ R_{m}(\xi_1) \cdot  \Omega_{m+1}(\xi_2') \in  \C/\Z(1).
\]
Note that for dimension and general position reasons alone,
\begin{equation}\label{E663}
T_{m+1}(\xi_1')\cdot T_m(\xi_2) = 0 = T_{m}(\xi_1)\cdot T_{m+1}(\xi_2') \in z_{\R}^{d+m+1}(X,2m+1),
\end{equation}
and likewise over $|\xi_1\cap \xi_2'|$ or $|\xi_1'\cap \xi_2|$,
\begin{equation}\label{E661}
\Omega_{m+1}(\xi_1') = 0 =  \Omega_{m+1}(\xi_2'),
\end{equation}
using the fact that $\dim |\xi_1\cap \xi_2'|, \ \dim |\xi_1'\cap \xi_2|
\leq m$ and that $\Omega_{m+1}(\xi_1')$, $\Omega_{m+1}(\xi_2')$
are meromorphic currents involving $m+1$ holomorphic differentials.
This, together with 
\[
R_{m}(\xi_1) \wedge \Omega_{m+1}(\xi_2') = \pr^{\ast}_{\square}(R_m\wedge \Omega_{m+1})\cdot
\delta_{\xi_1\cap \xi_2'},\ \text{\rm (by\ fiberwise\ Fubini)},
\]
implies (using (\ref{SSE})),  the simpler expression:
\[
\langle\xi_1,\xi_2\rangle_m^- := (-2\pi\I)^{m}{T_{m}(\xi_1)} \cdot R_{m+1}(\xi_2').
\]
Furthermore, the vanishing relations in (\ref{E663}) and (\ref{E661}) 
imply that the pairings $\langle \xi_1,\xi_2\rangle_m^{\pm}$ correspond
(up to twist) to $(\ast)$ in a Deligne complex triple of the form $(0,0,\ast)$,
(see the RHS of (\ref{E777})).
Note that if either $\partial\xi_1' = 0$ or $\partial \xi_2' = 0$,
then the pairings $\langle\xi_1,\xi_2\rangle_m^{\pm}$ amount to a cup product 
in Deligne cohomology of the regulator of 
a higher Chow cycle, together with one which is 
nullhomologous (in Deligne cohomology),
which is zero in:
\[
H_{\mathcal D}^{2d+1}(X,\Z(d+m+1)) \simeq \C/\Z(1),
\]
where firstly after incorporating the normalizing twist (just preceding (\ref{E777})), and
in our setting, we arrive at the isomorphisms:
\[
H_{\mathcal D}^{2d+1}(X,\Z(d+m+1)) \simeq  \C/\Z(d+m+1) {\buildrel \times(2\pi\I)^{d{+}m}\over \simeq} \C/\Z(1).
\]
Hence the pairings $\langle\xi_1,\xi_2\rangle_m^{\pm}$
do not depend on the choices of the $\xi_j'$'s.  
For simplicity, we will assume given
$(\xi_1,\xi_2)\in \Lambda(r,m,X)$. By definition,
this  implies that 
 \[
 \xi_1'\cap\xi_2,\
 \xi_1\cap\xi_2' \in z_{\R}^{d+m+1}(X,2m+1),\   \xi_1'\cap\xi_2'\in 
 z_{\R}^{d+m+1}(X,2m+2),
 \]
which is important in ensuring that the currents above are defined. 
Next, the relations
\[  
R_{2m+1}(\partial\{\xi_1'\cup \xi_2'\}) = R_{2m+1}(\xi_1\cup \xi_2') +
(-1)^{m+1}R_{2m+1}(\xi_1'\cup \xi_2),
\]
imply that
\begin{equation}\label{E522}
\langle\xi_1,\xi_2\rangle_m^+ = (-1)^m\langle\xi_1,\xi_2\rangle_m^-.
\end{equation}
We remark in passing that in the case $m=0$, and after taking real parts,
equation (\ref{E522}) implies the reciprocity result in Proposition \ref{PKK}.
The remaining claims in  Theorem \ref{MTA} are  left to
the reader.

 \end{proof}
 
 \begin{rmk} We can pass to a real-valued
 height pairing using the
 composite 
 $\C/\Z(1) \to \C/\R(1) \simeq \R$.
  \end{rmk}
 
We put
 \[
\langle\ ,\ \rangle_m := \langle\ ,\ \rangle_m^+\big|_{\Lambda(r,m,X)} = 
(-1)^m\langle\ ,\ \rangle_m^-\big|_{\Lambda(r,m,X)}.
\]
We will denote by $\langle\ ,\ \rangle_m^{\R}$ the corresponding
real pairing.

\bigskip

In the case $m=0$, we have $\Lambda^0(r,0,X) \subset z^r_{\rat}(X)\times
z^{d-r+1}_{\rat}(X)$. Let $(\xi_1,\xi_2)\in  z^r_{\rat}(X)\times
z^{d-r+1}_{\rat}(X)$. By considering the cases where $y = |\xi_j|$, $j=1,2$,
and regarding the real pairing $\langle\ ,\ \rangle_0^{\R}$ below,
we may assume that the domain is given by
\[
\Lambda := \big\{(\xi_1,\xi_2)\in z^r_{\rat}(X)\times
z^{d-r+1}_{\rat}(X)\ \big|\ |\xi_1|\cap |\xi_2| = \emptyset\big\}.
\]
and thus we have  pairings
\[
 \langle\ ,\ \rangle_0 : \Lambda^0(r,0,X) \to \C/\Z(1),
 \]
\[
\langle \ ,\ \rangle^{\R} := \langle\ ,\ \rangle^{\R}_0 : \Lambda  \to
\R,
\]
Let $\xi_1 := \DIV(f,D) \in z^r_{\rat}(X,0)$,
 $\xi_2 := \DIV(g,E)\in z_{\rat}^{d-r+1}(X,0)$ be given. In this
 case $D$ and $E$ are irreducible subvarieties of $X$
 of codim$_XD = r-1$ and codim$_XE = d-r$, and $f\in \C(D)^{\times}$,
 $g\in \C(E)^{\times}$. Then
 \[
  \langle\xi_1,\xi_2\rangle_0 = \int_{\{D\backslash f^{-1}[-\infty,0]\}\cap \xi_2}\log f  \in \C/\Z(1).
 \]
 Similarly,
 \[
 \langle\xi_1,\xi_2\rangle^{\R}_0 = \int_{D\cap \xi_2}\log|f| . 
 \]

\begin{rmk} It is instructive to work out
the case $m=1$.  Let 
 \[
 \xi_1 := \sum_j (g_j,D_j)\in z_{\rat}^{r}(X,1),
 \]
 and  
 \[
 \xi_2 := T(\{f_1,f_2\},E) \in z_{\rat}^{d-r+2}(X,1),
 \]
  be given, where
 $T$ is the Tame symbol. (We will also be working under the
 assumption that $(\xi_1,\xi_2)\in \Lambda(r,1,X)$.) In this
 case $E$ and $D_j$ are irreducible subvarieties of $X$
 of codim$_XD_j = r-1$ and codim$_XE_j = d-r$, and $f_i\in \C(E)^{\times}$,
 $i=1,2$, and  $g_j\in \C(D_j)^{\times}$. Set $C_j = D_j\cap E$, which is a curve
 in $X$, and put 
 \[
 g_{j,C_j} := g\big|_{C_j},\quad f_{i,C_j} = f_i\big|_{C_j}.
 \]
 Then using the identification $\langle \xi_1,\xi_2\rangle_1 = -
 \langle \xi_1,\xi_2\rangle_1^-$, a simple computation yields:
 \[
 \langle \xi_1,\xi_2\rangle_1 = \sum_j\biggl[(2\pi\I)\int_{g_{j,C_j}^{-1}[-\infty,0])}\log f_1 d\log f_2
 \]
 \[
 -\ (2\pi\I)^2\int_{(f_{1,C_j}\times g_{j,C_j})^{-1}[-\infty,0]^2}\log f_2\biggr]\  \in \C/\Z(3)
 \simeq \C/\Z(1).
 \]
 (Recall in equation (\ref{E666}) the identification
 $\C/\Z(2m+1) \simeq \C/\Z(1)$, which
 explains the need for the identification  $\C/\Z(3) \simeq \C/\Z(1)$
 in the case $m=1$.) Let $\gamma$ be the (closed) curve given by
 \[
 \gamma := \sum_jg_{j,C_j}^{-1}[-\infty,0].
 \]
 Taking the real part of $\langle \xi_1,\xi_2\rangle_1$
 and applying a Stokes' theorem argument, one
can show that:
\begin{equation}\label{234}
 \langle \xi_1,\xi_2\rangle_1^{\R} = -2\pi
\int_{\gamma}\big[\log|f_1|d\arg f_2  - \log|f_2|d\arg f_1\big].
 \end{equation}
 Equation  (\ref{234}) is easily seen to be non-trivial. [Take for example $X := E = \PP^2\ni [z_0,z_1,z_2]$,
 and consider $\PP^1 = \ell_j:= D_j := V(z_j)$, $g_0 = -z_1/z_2,\ g_1 = -z_2/z_0, \ g_2 = -z_0/z_1$. 
 Note that $g_j \in {\C(\ell_j)^{\times}}$ and that $\sum_{j=0}^2{\rm div}_{\ell_j} (g_j) = 0$.
 Put $L := z_0+z_1+z_2$, and $t_j = z_j/L$, hence $t_0+t_1+t_2=1$. 
 Thus in affine coordinates, $g_0 = -t_1/t_2,\ g_1 = -t_2/t_0,\ g_2 = -t_0/t_1$,
 and $\ell_j = V(t_j)$.
 Now let $\gamma$ be the corresponding $1$-cycle, which is the boundary of
the real simplex  $\{t_0+t_1+t_2  = 1\ |\ t_j\in [0,1]\}$. Let $\nu : \PP^2 \to \PP^1$
be the projection from $[0,0,1]$ (explicit: $\nu([z_0,z_1,z_2]) = [z_0,z_1])$). Then
$\nu(t_0,t_1,t_2) = (t_0,t_1)$ and the aforementioned real simplex becomes
$\{t_0+t_1\leq 1\ |\ t_j\in [0,1]\}$, and $\nu_*(\gamma)$ the
obvious boundary. Consider $p$ in the interior of $\nu_*(\gamma)$,
and in terms of the coordinate $w = t_0+\I t_1$, $t_j\in\R$, set 
$h(w) = w-p$. Observe that  $\int_{\nu_*(\gamma)}d\log h \ne 0$. Now put
\[
f_2 := \nu^*(h(w)) = \frac{z_0+\I z_1 +p\cdot L}{L},
\]
and  choose $f_1 \in \R^{\times}\bs \{\pm1\}$.]
\end{rmk}
\vspace{0.3cm}
\textbf{Acknowledgement:} The authors would like to thank Vincent Maillot for suggesting the idea
of a height pairing on graded pieces of a Bloch-Beilinson filtration, and to
Jos\'e Burgos Gil for providing us with the idea of the proof of Lemma \ref{IIILe1}. We are
also grateful from  correspondence with  Klaus K\"unnemann. A warm thanks goes to 
Lizhen Ji for his superb logistics in bringing this conference to fruition, and of course, to our esteemed colleague
Steven Zucker, for making this wonderful event possible.

\end{document}